\documentclass[onefignum,onetabnum]{siamart171218}
\usepackage[a4paper]{geometry}
\usepackage[utf8]{inputenc}
\usepackage[english]{babel}
\usepackage{lipsum}
\usepackage{booktabs}
\usepackage{amsfonts,amssymb}
\usepackage{graphicx}
\usepackage{epstopdf}
\usepackage{subcaption}
\usepackage{algorithmic}
\usepackage{amsmath}

\ifpdf
  \DeclareGraphicsExtensions{.eps,.pdf,.png,.jpg}
\else
  \DeclareGraphicsExtensions{.eps}
\fi

\usepackage{changes}

\newsiamremark{remark}{Remark}
\newsiamremark{hypothesis}{Hypothesis}
\crefname{hypothesis}{Hypothesis}{Hypotheses}
\newsiamthm{claim}{Claim}

\headers{How to optimize preconditioners for the conjugate gradient 
method}{A. Katrutsa, M. A. Botchev, G. Ovchinnikov and I. Oseledets}

\title{How to optimize preconditioners for the conjugate gradient 
method: a stochastic approach\thanks{Submitted to the editors 
\funding{This work was funded by RFBR grant 17-01-00854-a}}}

\author{Alexandr Katrutsa\thanks{Skolkovo Institute of Science and Technology 
  (\email{aleksandr.katrutsa@phystech.edu},
  \email{ovgeorge@yandex.ru}, \email{i.oseledets@skoltech.ru}).}
\and Mike Botchev\thanks{Keldysh Institute of Applied Mathematics, Russian Academy of Sciences, Moscow, Russia 
  (\email{botchev@ya.ru}).}
  \and George Ovchinnikov\footnotemark[2]
\and Ivan Oseledets\footnotemark[2]}

\usepackage{amsopn}

\makeatletter
\newcommand*{\addFileDependency}[1]{
  \typeout{(#1)}
  \@addtofilelist{#1}
  \IfFileExists{#1}{}{\typeout{No file #1.}}
}
\makeatother

\newcommand{\Gcg}{\mathcal{G}}
\newcommand{\Rr}{\mathbb{R}}
\newcommand{\Rrmm}{\Rr^{m\times m}}

\begin{document}

\maketitle

\begin{abstract}
The conjugate gradient method (CG) is typically used with a preconditioner which improves efficiency and robustness of the method.  
Many preconditioners include parameters and a proper choice of a preconditioner and its parameters is often not a trivial task.  
Although many convergence estimates exist which can be used for
optimizing preconditioners, these estimates typically
hold for all initial guess vectors, in other words,
they reflect the worst convergence rate.  
To account for the mean convergence rate instead,
in this paper, we follow a  stochastic approach. 
It is based on trial runs with random initial guess vectors
and leads to a functional which can be used to monitor
convergence and to optimize preconditioner parameters in CG.
Presented numerical experiments show that optimization of this 
new functional with respect to preconditioner parameters usually 
yields a better parameter value than optimization of the functional based on the spectral condition number. 
\end{abstract}



\section{Introduction}
Preconditioning is an important tool for improving convergence while solving linear systems iteratively~\cite{AxelssonBook,SaadBook2003,Henk:book}.
Efficient preconditioners typically do not only improve the condition number 
of the system matrix but, more importantly, lead to clustering of its eigenvalues.
In nonstationary methods, like the 
conjugate gradient method (CG)~\cite{gvl,hestenes1952methods} or the generalized minimal residual method  (GMRES)~\cite{SaadBook2003,saad1986gmres},
this improved clustering usually manifests in superlinear 
convergence~\cite{RateCG,SupGMRES} (at least in the case of a normal matrix).

Preconditioners may often
include parameters and, although major preconditioner classes have been extensively analyzed~\cite{AxelssonBook,SaadBook2003,Henk:book}, a practical choice of their parameters is not a trivial task.
To optimize preconditioners  parameters in practice, different target functionals can be used~\cite{Benzi2002survey}: the spectral radius~\cite{Varga_book}, 
Ritz values~
\cite{Dongarra_NLA_HPC1998}
of the preconditioned matrix, the so-called K-condition 
number~\cite{Kaporin1994,Kaporin2010}, 
a suitable norm of the iteration 
matrix~\cite{BotchevKrukier1997,BotchevGolub2006}, 
closeness in the Frobenius norm of the preconditioned matrix to the 
identity matrix~\cite{Benzi2002survey},
and the trace of the preconditioned
matrix~\cite{Benzi_ea2016}.
All these functionals reflect the convergence behavior of preconditioned 
iterations and have one common feature: 
they are based on certain convergence estimates, which hold for \emph{all} possible initial guess vectors.  
In this sense, they represent a worst-case scenario and, hence, a question arises whether these functionals are adequate for choosing preconditioner  parameters in practice.  
Would, for example, monitoring the \emph{mean} convergence rate be a better option rather than the \emph{worst-case} rate?

This paper presents an attempt to answer this question.
A simple convergence analysis shows
that a faster convergence is observed for a nonempty open set of initial guess vectors.
This suggests that the mean convergence rate can indeed be 
a more adequate convergence measure than the worst-case rate.
Furthermore, we present a stochastic optimization approach based on 
trial runs with random initial guess vectors.  
This leads to an optimization functional which can be used to monitor convergence and to optimize preconditioner parameters in CG. 
We show in numerical experiments that optimization of this new functional with 
respect to preconditioner parameters usually gives a better 
parameter value than optimization of the functional based on the condition number. 

This is confirmed in the first numerical test, where we show how the stochastic convergence functional can be
used to optimize 
the well-known relaxed incomplete Cholesky preconditioner 
without fill in, which we denote by $\mathrm{RIC}_{\alpha}(0)$~\cite{axelsson1986eigenvalue,Henk:book,Meurant_book1999}.
This is done for linear systems stemming from
a finite-difference approximation of diffusion problems.
In the second numerical test we consider a structural
mechanics problem solved by CG in combination with the 
SSOR($\omega$) preconditioner.  
Estimates for optimal values of $\omega$ 
in the SSOR preconditioner are well known
but often are expensive to compute~\cite{Templates,gvl,AxelssonBook}.
For instance, this is the case for the optimal value for the SOR 
method (which holds true for consistently ordered matrices
with property $A$~\cite{Young_book})
\begin{equation}
\label{om_opt}
\omega = \dfrac{2}{1+\sqrt{1-\rho_J^2}},
\end{equation}
where $\rho_J$ is the spectral radius of the iteration
matrix of the Jacobi iterative scheme.
This formula might also give a reasonable
suboptimal value for the SSOR preconditioner.  
However, for this test problem the Jacobi iterations
do not converge so that the formula above does
not make sense. 
On the other hand, we demonstrate that the stochastic 
convergence functional is able to provide an optimal
$\omega$ value in this test problem.

We emphasize that the suggested approach is currently of restricted
practical value if just a single linear system has to be solved.
This is because the proposed optimization procedure is based on 
trial runs with many random initial guess vectors and
implies significant computational costs.
However, there are situations where many linear systems with 
the same system matrix have to be solved, e.g. in implicit time integration of large dynamical systems,
and the preconditioner optimization can lead to a 
significant convergence improvement.  
In such cases our approach can be practically useful.

We note that solving many linear systems with the same
system matrix has been an active research 
direction, see~\cite{Amritkar_ea2015,BennerFeng2011}.  
The contribution of our  stochastic optimization approach here is that it can be used in combination with these techniques, leading to further savings of computational costs by an appropriate choice of preconditioner parameters.

Throughout this paper, we assume that a linear system 
\begin{equation}
\label{lin_syst}
Ax = b
\end{equation}
has to be solved for a given symmetric positive definite matrix
$A\in\Rrmm$ and many different right hand side vectors $b\in\Rr^m$. 
The rest of the paper is organized as follows. 
In Section~\ref{sect:functional} the stochastic convergence functional is introduced for stationary linear iterative methods, i.e., for iterations of the form
\begin{equation}
\label{lin_m}
Mx_{k+1} = Nx_k + b,
\end{equation}
with $M-N=A$ and nonsingular $M$.
In the same section, we also introduce a similar convergence functional for 
nonstationary nonlinear iterations such as CG.
A question whether the proposed convergence functional provides a convergence measure different than a classical convergence estimate based on the spectral condition number is discussed in Section~\ref{sect:conv}.
There we show that an open set of initial guess vectors 
exists for which CG converges faster than predicted by the classical estimate.  
This confirms that our proposed convergence functional is an essentially different convergence measure than the classical one.
Finally, numerical experiments are presented in Section~\ref{sect:numexp} and conclusions are drawn in Section~\ref{sect:concl}.



\section{Mean convergence rate}
\label{sect:functional}
Iterative solvers for linear systems are well  
studied~\cite{AxelssonBook,gvl,greenbaum1997iterative,Meurant_book2006,SaadBook2003,Henk:book}
and many classical convergence estimates are available.
A~convergence estimate typically has the form 
\begin{equation}
\label{newconv:upper}
\Vert x_* - x_k \Vert_* \leq C q_A^k \Vert x_* - x_0 \Vert_*,
\end{equation}
where $x_*$ is the exact solution vector of~\eqref{lin_syst}, $x_k$ is the $k$-th iterand of the method, $\|\cdot\|_*$ is some vector norm, $C>0$ is a constant and $q_A>0$ is a constant depending on the matrix $A$. 
The estimate~\eqref{newconv:upper} is a worst-case estimate among \emph{all} initial guess vectors $x_0$, whereas it is quite natural to study \emph{mean convergence rate} of a given iterative method instead. 

\begin{definition}
Let mean convergence rate be upper estimation of the norm $\| x_* - x_k \|_*$ averaged over all possible initial guess vectors $x_0$, i.e $\mathbb{E}_{x_0} (\Vert x_* - x_k \Vert_*^2)$, where $x_0$ is generated from some given distribution.
\end{definition}

In this paper we want to study this approach which, as far as we know, has not been explored in this way. 
We consider the initial error vector $x_*-x_0$ to be a \emph{random vector} with independent and identically distributed (i.i.d.) entries with a mean of 0 and a standard deviation 
of 1, i.e., in $N(0, 1)$. 
Then the error $x_* - x_k$ is also a random vector, and we can define its expectation $e_k$ by
\begin{equation}
\label{newconv:exp}
e_k^2 = \mathbb{E} (\Vert x_* - x_k \Vert_*^2),
\end{equation}
where expectation is taken over the distribution of $x_* - x_k$. 
In general one can not expect that the entries of $x_* - x_k$ are distributed normally; their distribution is unknown since the vector $x_* - x_k$ results from a nonlinear CG process\footnote{In Section~\ref{sec::gf} we consider non-normally distributed $x_* - x_0$.}.
Therefore we further use a computable unbiased estimation~\eqref{Fs} of the expectation in~\eqref{newconv:exp}.
A question arises whether an estimate of the form
\begin{equation}
\label{newconv:exp2}
   e_k \sim C \mu_A^k
\end{equation}
can be obtained,
where $\mu_A>0$ is a constant which depends on $A$ and
determines the convergence rate.
We should be careful while giving a meaning to the \emph{asymptotic} behavior 
in~\eqref{newconv:exp}, since for some methods (e.g., CG) in exact arithmetic we have 
convergence after $m$ iterations. Nevertheless, for large system dimension $m$ the estimates of the form~\eqref{newconv:exp} are of interest and provide useful information about convergence. 

First, the important special case is the stationary linear iterative method~\eqref{lin_m},
\begin{equation}
    x_*- x_{k+1} = G (x_* - x_k), \qquad G = M^{-1}N, 
    \label{eq::lin_iter_stat}
\end{equation}
where $G$ is the iteration matrix and $M-N=A$.
Then, using the classical Hutchinson result on stochastic 
trace estimator~\cite{Hutchinson}
and denoting $d_k=x_*-x_k$, $k\geqslant 0$,
we obtain in the 2-norm
\[
e_k^2 = \mathbb{E}\left( G^k d_0, G^k d_0\right) = \mathbb{E} \left((G^k)^{\top} G^k  d_0, d_0 \right) = \mathrm{Tr}((G^k)^{\top} G^k) = ||G^k||^2_F,
\]
where $\|\cdot\|_F$ is the Frobenius matrix norm 
and $\mathrm{Tr}$ denotes the trace of a matrix. 
Due to Gelfand's formula
\[
\lim_{k\rightarrow \infty} \Vert G^k \Vert_*^{1/k} = \rho(G),
\]
where $\|\cdot\|_*$ is any norm\footnote{The norm
does not have to be an operator norm (i.e., induced by a vector norm) 
but if so the limit is 
approached from above.} and $\rho(G)$ denotes the spectral radius of $G$, 
we have
\begin{equation}
\label{e=rho}
e_k \sim \rho(G)^k,
\end{equation}
i.e., in this case the worst-case rate is also the mean rate~\cite{katrutsa2017deep}.

For nonlinear nonstationary iterative methods, such as CG or MINRES, similar analysis appears to be quite complicatedand is left beyond the scope of this paper. 
However, we find experimentally, by Monte-Carlo simulations and fitting estimated convergence rates, that the situation is completely different, i.e., the worst-case rates are significantly larger than the estimated mean convergence rates. 
Since analytical expressions for the mean convergence  rate are not available, we can follow a practical approach and try to derive a computable measure of convergence similar to~\eqref{e=rho}.
As $k$~iterations of a stationary linear method~\eqref{lin_m} are carried out through $k$~matrix-vector multiplication with the iteration matrix~$G$, similarly, $k$~iterations of a nonstationary method can be seen as an action of some nonlinear mapping (defining the method) $k$~times. 

Then, a straightforward, practical way to monitor convergence of the method is to perform $k$ iterations for a number of random initial
guess vectors $x_0^{(i)}$.  Indeed, let us define 
a stochastic convergence functional
\begin{equation}
\label{Fs}
F_s \equiv \frac1{n}
\sum_{i=1}^{n}\|x_* - x_k^{(i)}\|_*,
\end{equation}
where $n$ is the number of random initial guess vectors and $x_k^{(i)}$ is the $k$-th iterand of the method started at initial guess vector $x_0^{(i)}$.
Note that, analogously to~\eqref{eq::lin_iter_stat}, we can write
\begin{equation*}
F_s = \frac{1}{n} \sum_{i=1}^{n}\|x_* - x_k^{(i)}\|_*
= \frac{1}{n}\sum_{i=1}^{n}\|\Gcg^k(x_* - x_0^{(i)})\|_*
\approx \mathbb{E}_{x_0} \|\Gcg^k(x_* - x_0)\|_*,
\end{equation*}
where $\Gcg$ is a non-linear mapping corresponding to one iteration of CG and $\Gcg^k$ denotes the mapping applied $k$~times.

A possible application of the introduced concept of the mean convergence rate is optimization of the preconditioner parameters in preconditioned CG method to get faster convergence.
To perform optimization we introduce in~\eqref{Fs} a preconditioner  parameter~$\alpha$ since $x_k^{(i)}$ is obtained after $k$ iterations of the preconditioned CG and depends on~$\alpha$. 
Hence, the mapping $\Gcg$ depends on the preconditioner parameter and, formally speaking, we optimize the functional 
\[
F_s(\alpha) = \frac{1}{n}\sum_{i=1}^{n}\|\Gcg^k(x_* - x_0^{(i)}, \alpha)\|_*
\approx \mathbb{E}_{x_0} \|\Gcg^k(x_* - x_0, \alpha)\|_*.
\]
This means that we actually optimize the functional~\eqref{Fs} with respect to the parameter $\alpha$.
Evidently, such an optimization process based on many trial runs is expensive. 
One trial run means carrying out
$k\cdot n$ iterations of the method.
These costs are only paid off if the optimized preconditioner
is to be used in many iterations, for example, if many 
linear systems with the same matrix and different right
hand sides have to be solved.  This is why
we make this assumption while introducing~\eqref{lin_syst}.
More details on practical optimization of functional $F_s$ can be found in Section~\ref{sec::opt_cost}.

Since for stationary iterations the stochastic convergence functional 
appears to be identical to a classical convergence measure i.e., the spectral radius of the iteration matrix,
a question arises whether our stochastic convergence functional~$F_s$ does
not coincides with some known convergence measure
for nonstationary iterations.
In the next section we study this question for the CG method.
A well-known classical convergence estimate for 
CG, see e.g.~\cite{AxelssonBook,gvl,Meurant_book2006,SaadBook2003,Henk:book}, 
is based
on the condition number $\kappa$ of the system matrix $A$:
\begin{equation}
\label{class_est}
\|x_*-x_k\|_A \leqslant
2\left(\dfrac{\sqrt{\kappa}-1}{\sqrt{\kappa}+1}\right)^k
\|x_*-x_0\|_A.
\end{equation}
Although this estimate can in general be pessimistic, as it does not reflect
the often observed superlinear convergence of CG~\cite{RateCG,axelsson1986eigenvalue},  
it can still be used for monitoring the convergence rate of CG.
%
Hence, together with~\eqref{Fs}, we consider the corresponding classical 
convergence functional,~i.e.,
\begin{equation}
\label{Fc}
F_c\equiv \left(\dfrac{\sqrt{\kappa}-1}{\sqrt{\kappa}+1}\right)^k.
\end{equation}
Note that~\eqref{class_est} can be improved if a clustering of the eigenvalues
is assumed~\cite{axelsson1986eigenvalue,SaadBook2003,RateCG}.
In the next section we show that even if no assumptions
on eigenvalue clustering are made, 
there is an open set of initial guess vectors for which 
CG exhibits a faster convergence than predicted by the classical estimate~\eqref{class_est}.
This implies that our stochastic functional~\eqref{Fs} is essentially different than the classical convergence functional~\eqref{Fc}.

\section{Initial guess vectors and convergence of CG}
\label{sect:conv}
Analysis in this section is inspired by results of~\cite{RateCG}. 
Throughout this section we neglect the round off errors.
Let $A$ be a symmetric positive definite $m\times m$ matrix, $z_1$, \dots, $z_m$ be orthonormal eigenvectors of $A$, and $0<\lambda_1\leqslant\dots\leqslant\lambda_m$ be the corresponding eigenvalues.
For simplicity of notation throughout this section we omit the subscript $\cdot_*$ in 
the exact solution vector $x_*$.
\newcommand{\polset}[1]{\Pi_{#1}^{\text{ct1}}}
For the CG iterands $x_j$, $j=1,2,\dots$, the optimality property reads
\begin{equation}
\label{opt1}
\|x-x_j\|_A = \min_{q\in\polset{j}}\|q(A)(x-x_0)\|_A,  
\end{equation}
where $\polset{j}$ denotes the set of all polynomials of degree
at most $j$ with the constant term~$1$.
Let $q$ be the CG residual polynomial (i.e., the polynomial
at which the minimum in~\eqref{opt1} is attained),
and let
\[
x-x_0 = \sum_{i=1}^m\gamma_i z_i.
\]
It is easy to check that the optimality property~\eqref{opt1}
can be rewritten as
\begin{equation}
\label{est0}
\|x-x_j\|_A^2 = \sum_{i=1}^m \lambda_i(\gamma_iq(\lambda_i))^2
\leqslant \sum_{i=1}^m \lambda_i(\gamma_i\tilde{q} (\lambda_i))^2, \qquad
\forall \tilde{q}\in\polset{j}.
\end{equation}
Moreover, we have
\begin{equation}
\label{respol}
q(t) = \dfrac{(\theta_1-t)\dots(\theta_j-t)}{\theta_1\dots\theta_j},
\end{equation}
where the roots $\theta_1$, \dots, $\theta_j$ of $q(t)$ are the Ritz values
of the CG process at the $j$-th iteration.  

Furthermore, consider $\bar{x}_0$ chosen such that
\begin{equation}
\label{x0}
x-\bar{x}_0 = \sum_{i=2}^m\gamma_i z_i,
\end{equation}
and denote by $\bar{q}(t)$ the CG residual polynomial of the
CG iterations with $\bar{x}_0$ taken as the initial guess.
Similarly to~\eqref{respol}, it holds
\[
\bar{q}(t) = \dfrac{(\bar{\theta}_1-t)\dots(\bar{\theta}_j-t)}{\bar{\theta}_1\dots\bar{\theta}_j},
\]
with $\bar{\theta}_1$, \dots, $\bar{\theta}_j$ being the Ritz values of the CG process
started at $\bar{x}_0$.
Note that by taking in~\eqref{est0} the polynomial $\tilde{q}(t)$ as the Chebyshev
minimax polynomial on the interval $[\lambda_1,\lambda_m]$, we obtain
the classical convergence estimate:
\begin{equation}
\label{est1}
\begin{aligned}
\|x-x_j\|_A^2 &\leqslant \sum_{i=1}^m \lambda_i(\gamma_i\tilde{q} (\lambda_i))^2
\\
&\leqslant\max_{i}\tilde{q}(\lambda_i)^2\sum_{i=1}^m \lambda_i\gamma_i^2
\leqslant\max_{\lambda\in[\lambda_1,\lambda_m]}\tilde{q}(\lambda)^2\sum_{i=1}^m \lambda_i\gamma_i^2
\\
&= \max_{\lambda\in[\lambda_1,\lambda_m]}\tilde{q}(\lambda)^2\cdot\|x-x_0\|_A^2
= 4C_1^{2j}\|x-x_0\|_A^2,
\end{aligned}
\end{equation}
where
\[
C_1 = \dfrac{\sqrt{\kappa_1}-1}{\sqrt{\kappa_1}+1}, \qquad \kappa_1 = \dfrac{\lambda_m}{\lambda_1}.
\]
For the CG process started at $\bar{x}_0$ the corresponding convergence estimate reads
\begin{equation}
\label{est2}
\|x-\bar{x}_j\|_A^2 \leqslant 4C_2^{2j}\|x-\bar{x}_0\|_A^2,\qquad
C_2 = \dfrac{\sqrt{\kappa_2}-1}{\sqrt{\kappa_2}+1}, \qquad \kappa_2 = \dfrac{\lambda_m}{\lambda_2}.
\end{equation}

\begin{theorem}
\label{T1}
Let the initial guess vector $x_0$ in the CG process be chosen
such that the first component $\gamma_1$ of the initial error $x-x_0$
is small with respect to the other components;
more precisely, let there
exist a constant $\delta>0$ such that
\begin{equation}
  \label{small_g1}
  \dfrac{\lambda_1\gamma_1^2}{\displaystyle\sum_{i=2}^m\lambda_i\gamma_i^2}
  =
  \dfrac{\lambda_1\gamma_1^2}{\|x-\bar{x}_0\|_A^2}
  \leqslant 4 C_2^{2j} \delta, \quad \text{for } j=1,\dots,J,
\end{equation}
where $x_0$ is defined in~\eqref{x0}.
Then convergence of the CG process in the first $J$ iterations
is determined by the constant $C_2$ rather than by $C_1$
(cf.~\eqref{est1},\eqref{est2}) in the sense that
\begin{equation}
\label{est3}
\|x-x_j\|_A^2 \leqslant
4(1+\delta)C_2^{2j}\|x-x_0\|_A^2,
\quad \text{for } j=1,\dots,J.
\end{equation}
\end{theorem}

\begin{proof}
Choosing in~\eqref{est0} the polynomial~$\tilde{q}(t)$ as the residual
polynomial~$\bar{q}(t)$ of the CG process started at~$\bar{x}_0$, we
have
\begin{align*}
\|x-x_j\|_A^2
&=
\sum_{i=1}^m\lambda_i(\gamma_i q(\lambda_i))^2
\leqslant
\sum_{i=1}^m\lambda_i(\gamma_i\bar{q}(\lambda_i))^2
\\
&=
\lambda_1(\gamma_1\bar{q}(\lambda_1))^2
+\sum_{i=2}^m\lambda_i(\gamma_i\bar{q}(\lambda_i))^2
\\
&=\lambda_1(\gamma_1\bar{q}(\lambda_1))^2
+\|x-\bar{x}_j\|_A^2
\leqslant\lambda_1(\gamma_1\bar{q}(\lambda_1))^2
+4C_2^{2j}\|x-\bar{x}_0\|_A^2
\\
&\leqslant\lambda_1\gamma_1^2
+4C_2^{2j}\|x-\bar{x}_0\|_A^2,
\end{align*}
where the last inequality holds because
\[
0\leqslant \bar{q}(t)\leqslant 1,
\]
which is true since $\bar{q}(t)$ is
monotonically non-increasing on the interval $[0,\lambda_2]$
and $\bar{q}(0)=1$.
Finally, we use assumption~\eqref{small_g1} and obtain
\begin{align*}
\|x-x_j\|_A^2
&\leqslant
\lambda_1\gamma_1^2 +4C_2^{2j}\|x-\bar{x}_0\|_A^2
\\
&\leqslant
\delta 4C_2^{2j}\|x-\bar{x}_0\|_A^2
+
4C_2^{2j}\|x-\bar{x}_0\|_A^2
\leqslant
4(1+\delta)C_2^{2j}\|x-x_0\|_A^2.
\end{align*}
This ends the proof.
\end{proof}

It is not difficult to see that the last theorem can be generalized
for the case where several first components of the initial error
are small with respect to the other error components.  Indeed, denote
\begin{equation}
\label{Cs}
C_s = \dfrac{\sqrt{\kappa_s}-1}{\sqrt{\kappa_s}+1}, \qquad
\kappa_s = \dfrac{\lambda_m}{\lambda_s}.  
\end{equation}
Then the following result holds.

\begin{theorem}
\label{T2}
Let the initial guess vector $x_0$ in the CG process be chosen
such that the first $s-1$ components $\gamma_1$, \dots, $\gamma_{s-1}$
of the initial error $x-x_0$ are small with respect to the other components;
more precisely, let there
exist a constant $\delta>0$ such that
\begin{equation}
  \label{small_gs}
  \dfrac{\displaystyle\sum_{i=1}^{s-1}\lambda_i\gamma_i^2}
        {\displaystyle\sum_{i=s}^m\lambda_i\gamma_i^2}
  =
  \dfrac{\displaystyle\sum_{i=1}^{s-1}\lambda_i\gamma_i^2}
        {\|x-\bar{x}_0\|_A^2}
  \leqslant 4 C_s^{2j} \delta, \quad \text{for } j=1,\dots,J,
\end{equation}
where
\begin{equation}
\label{x0a}
x-\bar{x}_0 = \sum_{i=s}^m\gamma_i z_i,
\end{equation}
Then convergence of the CG process in the first $J$ iterations
is determined by the constant $C_s$ rather than by $C_1$
(cf.~\eqref{est1}) in the sense that
\begin{equation}
\label{est4}
\|x-x_j\|_A^2 \leqslant
4(1+\delta)C_s^{2j}\|x-x_0\|_A^2,
\quad \text{for } j=1,\dots,J.
\end{equation}
\end{theorem}

\begin{proof}
The proof is analogous to the proof of Theorem~\ref{T1}.
We take in~\eqref{est0} polynomial~$\tilde{q}(t)$ as the residual
polynomial~$\bar{q}(t)$ of the CG process started at~$\bar{x}_0$
defined in~\eqref{x0a}.  Using the convergence estimate
$$
\|x-\bar{x}_j\|_A^2 \leqslant 4C_s^{2j}\|x-\bar{x}_0\|_A^2,
$$
which holds for this CG process, and the assumption~\eqref{small_gs},
we arrive at~\eqref{est4}.
\end{proof}

\begin{figure}
\centering
\includegraphics[width=0.6\linewidth]{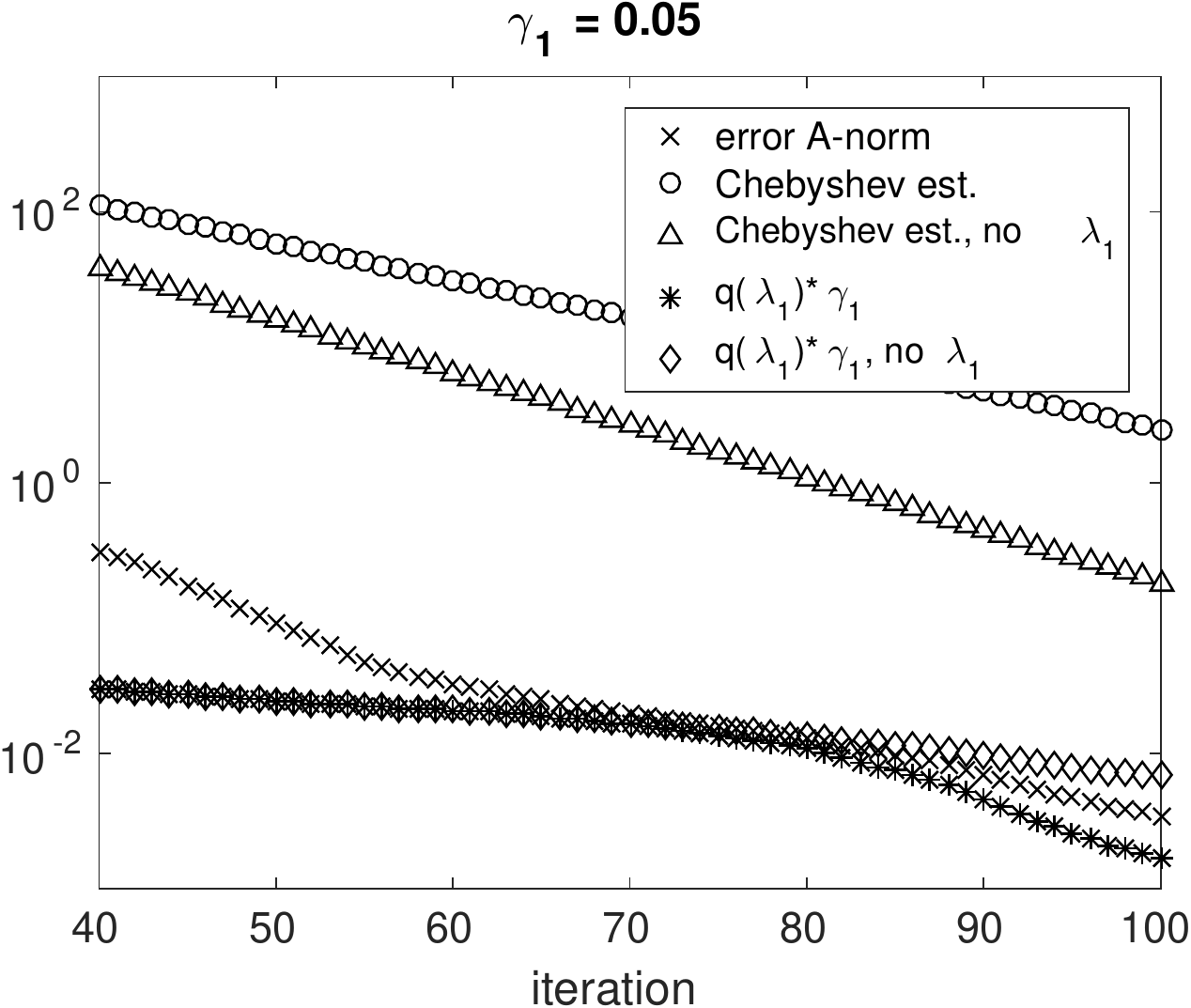}
\\[2ex]
\includegraphics[width=0.6\linewidth]{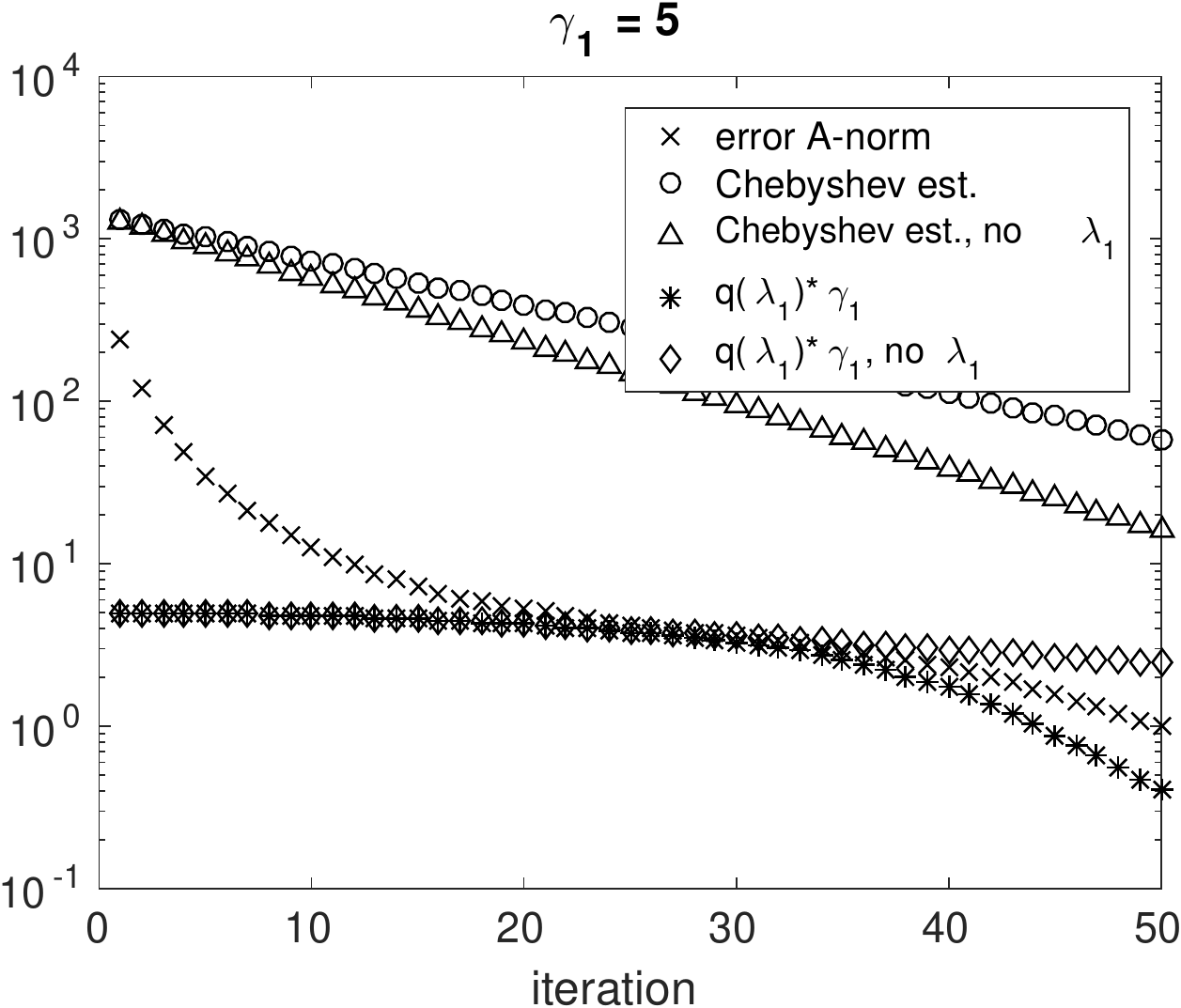}
\caption{CG convergence for the initial guess vectors with $\gamma_1=0.05$ (top) and $\gamma_1=5$ (bottom):
  the error norm $\|x-x_j\|_A$ (the $\times$ curve),
  estimate $2C_1^j$ (the $\circ$ line), estimate $2C_2^j$ (the $\triangle$ line),
and the values $\gamma_1q(\lambda_1)$ and 
$\gamma_1\bar{q}(\lambda_1)$ (the $*$ and $\diamond$ curves, respectively).}
\label{fig1}  
\end{figure}

Theorems~\ref{T1},~\ref{T2} can be illustrated by the following numerical test.
Let $A$ be a diagonal matrix of dimension $m=1000$, with the diagonal entries
\[
1, 2, 3, \dots, 1000.
\]
Let, furthermore, the right hand side vector $b$ be taken such that the exact 
solution vector has all its components one.  
The initial guess vector $x_0$ is chosen such all the components of $x-x_0$ except the first one are 
i.i.d.\ and in $N (0, 1)$ (independent
normally distributed random values with zero mean and variance one).
The first entry of $x-x_0$ is set to $\gamma_1$.

\begin{figure}
\centering
\includegraphics[width=0.6\linewidth]{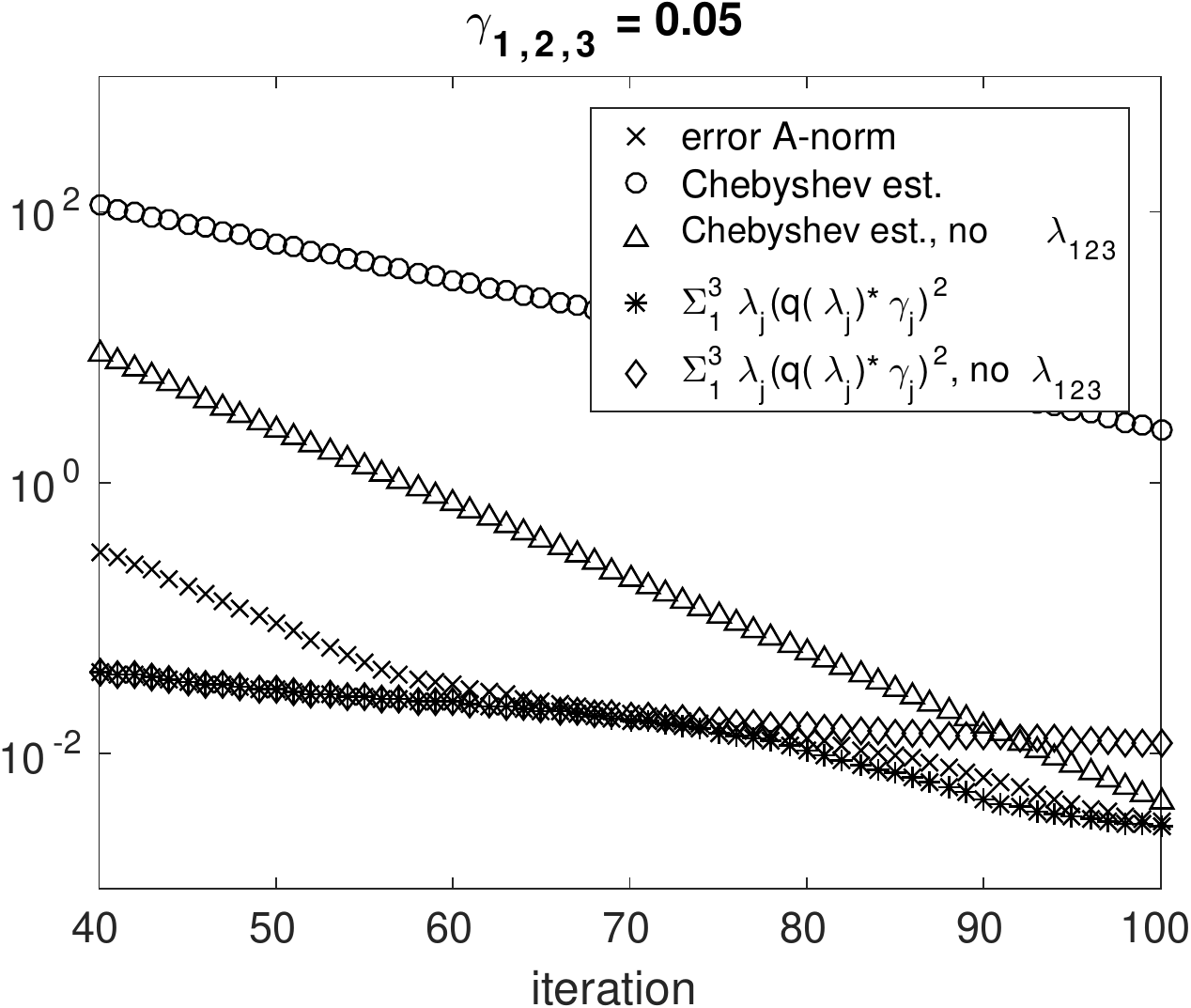}
\\[2ex]
\includegraphics[width=0.6\linewidth]{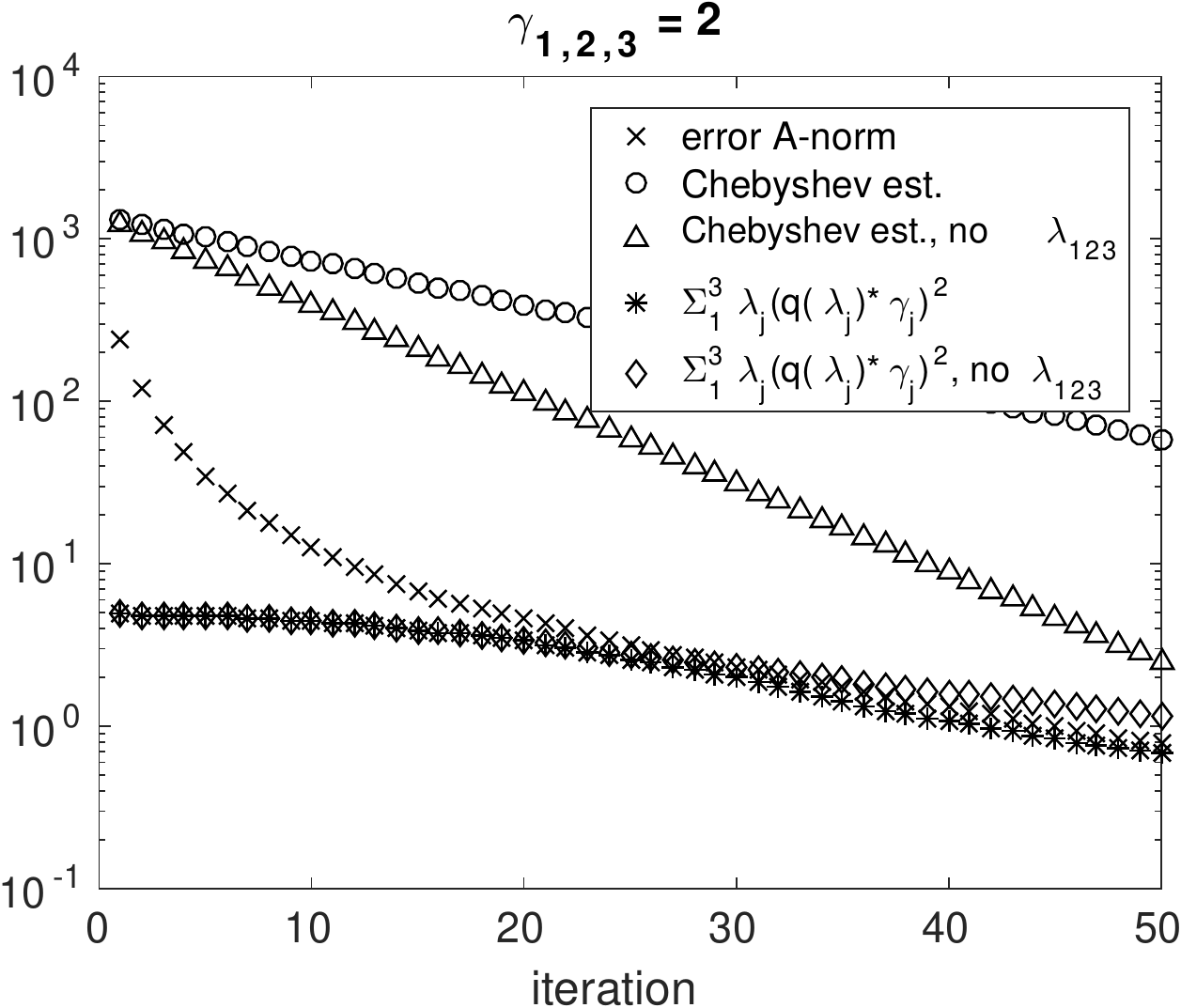}
\caption{CG convergence for the initial guess vectors with 
$\gamma_1=\gamma_2=\gamma_3=0.05$ (top) and
$\gamma_1=\gamma_2=\gamma_3=2$ (bottom):
the error norm $\|x-x_j\|_A$ (the $\times$ curve),
estimate $2C_1^j$ (the $\circ$ line), estimate $2C_2^j$ (the $\triangle$ line),
and the values $\gamma_1q(\lambda_1)$ and 
$\gamma_1\bar{q}(\lambda_1)$ (the $*$ and $\diamond$ curves, respectively).}
\label{fig2}  
\end{figure}
 
In Figure~\ref{fig1} the CG error convergence is plotted, together with the Chebyshev bounds~\eqref{est1}, \eqref{est2} and the values $\gamma_1q(\lambda_1)$, $\gamma_1\bar{q}(\lambda_1)$.
As we see, at first iterations (approximately until iteration~$75$ for
$\gamma_1=0.05$ and iteration~$25$ for $\gamma_1=5$) the values $\gamma_1q(\lambda_1)$ and  $\gamma_1\bar{q}(\lambda_1)$ are practically equal and stay
almost constant.  
This means that CG converges just as if the first error component were absent.
As clearly seen in the first plot of Figure~\ref{fig1}, up to iteration $75$,  the slope of the error $A$-norm (the $\times$ curve) is determined by the improved Chebyshev estimate (the $\triangle$ line), which confirms the estimate~\eqref{est3}. 
The first error component is ignored by the CG iterations until it becomes comparable in magnitude with the total error norm (until the $\times$ curve crosses the $*$ curve).  
Starting from this point, $\gamma_1q(\lambda_1)$ starts to decrease, thus damping the first error component.  
The value of $\gamma_1\bar{q}(\lambda_1)$ keeps on staying almost unchanged and exceeds the error norm.
It is interesting to note that at the cross point of the $\times$ line and the $*$ line (corresponding to iteration~$75$ for $\gamma_1=0.05$ and iteration~$25$ for $\gamma_1=5$) the estimate~\eqref{small_g1} holds for approximately the same values of $\delta$, namely for $\delta<10^{-3}$.
Thus, this value can be seen as a ``threshold'' value for what CG process ``considers'' as small.

Results of a similar test are presented in Figure~\ref{fig2}.
The parameters of test runs are the same, except that the first three components of the initial error $\gamma_1$, $\gamma_2$, $\gamma_3$
are now given the same specific value (namely, either $0.05$ or $2$).
Accordingly, the comparison CG process with the polynomial $\bar{q}(t)$
(plot by the $\diamond$ curve) is
now started at $x_0$, as defined in~\eqref{x0a} with $s=4$.

\begin{remark}
We emphasize that our results in this section are different from the ``effective condition number'' concept in the sense that we do not assume that some components in the initial error $x - x_0$ are zero.  
Furthermore, we note that our convergence results can be seen as a complement to the classical convergence estimates of Van~der~Sluis and Van~der~Vorst~\cite{RateCG} in the following sense.
Our results specify possible convergence behavior of CG at the initial stage, i.e., before certain components in the error are damped and CG exhibits its well known superlinear convergence (this ``superlinear''
convergence phase 
is seen in the $\times$ curves in Figures~\ref{fig1} and~\ref{fig2}
after they cross the $\diamond$ curves).
\end{remark}

\section{Practical optimization procedure}
\label{sec::opt_cost}
%
To optimize preconditioner parameters with respect to the stochastic convergence functional $F_s$, we use Brent method~\cite{brent2013algorithms}, which assumes multiple evaluations of $F_s$. 
This optimization method, which is a combination of the golden search and inverse quadratic interpolation,
requires a single evaluation of the functional per optimization step.
According to~\eqref{Fs} evaluation of $F_s$ requires prior knowledge of the unknown~$x_*$.
However, for practical evaluation of $F_s$ during the optimization procedure we set $x_*$ to be zero. 
This approach gives a tractable way to compute and therefore optimize $F_s$ with respect to the preconditioner parameter. 
The optimal preconditioner parameter is then used in solving test problems with arbitrary nonzero $x_*$.


Hence, the total costs of the optimization procedure can approximately be expressed as $K n s$ preconditioned 
matrix--vector products (matvecs), 
where $K$ is the number of the preconditioned iterations
(its choice is discussed below in Section~\ref{sect:2func}),
$n$ is the number of random initial guess vectors,
$s$ is the number of optimization steps needed to find the optimal value to an acceptable accuracy.

Thus, our optimization procedure is more efficient than a simple
trial-and-error search provided that the number of optimization
steps~$s$ is smaller than the number of trial-and-error runs.
In numerical experiments presented below we observe that up to $s=25$ optimization
steps suffice to achieve optimization accuracy $10^{-5}$, whereas to find the optimal value by trial-and-error runs
usually approximately $100$ test runs are needed to achieve the same optimization accuracy.

The same optimization procedure with Brent's method
is used in numerical tests of Section~\ref{sect:numexp} to
optimize the classical condition number functional $F_c$.
To compute the condition number in $F_c$ we use a standard sparse eigenvalue solver 
of Python numerical library
(this eigenvalue solver is 
similar to the \texttt{eigs} command in Octave and Matlab 
and based on the ARPACK software~\cite{ARPACK}).
We note that using eigensolvers for evaluating $F_c$ 
may be prohibitively expensive in practice and is done
only to compare parameter optimization based on~$F_s$
and on~$F_c$.


\section{Numerical experiments}
\label{sect:numexp}
In this section we present comparison of the classical functional and the proposed stochastic one for choosing an optimal parameter in preconditioners for two test problems.
The first one is a diffusion problem solved by CG with the $\mathrm{RIC}_{\alpha}(0)$ preconditioner~\cite{axelsson1986eigenvalue}.
The second test problem is a mechanical structure problem \texttt{bccst16} from The SuiteSparse Matrix Collection~\cite{davis2011university}, where CG is preconditioned by SSOR($\omega$)~\cite{Templates}.

\subsection{Test problem 1}
\label{sec::test_problem1}

In this test problem linear systems are obtained by the standard second order central 
finite difference approximation of the following Dirichlet boundary value
problem for unknown $u(x,y)$:
\begin{equation}
\begin{aligned}
& - (D_1 u_x)_x - (D_2 u_y)_y = g(x,y), \quad (x,y)\in\Omega=[0,1]\times[0,1],
\\
& u(x,y)|_{\partial\Omega}=0,
\end{aligned}
\label{eq:model_problem}
\end{equation}
where the subscripts $\cdot_{x,y}$ denote the partial derivatives 
with respect to $x$ and $y$.
We consider two cases: in the first case the coefficients $D_{1,2}$ are 
taken to be identically one in the whole domain~$\Omega$.  In the second
case the coefficients $D_{1,2}$ are discontinuous:
\[
D_1 = \begin{cases}
1000, & \quad (x,y) \in \left[\frac14,\frac34\right]\times \left[\frac14,\frac34\right],
\\
1, & \quad \text{otherwise},
\end{cases}
\qquad
D_2 = \frac12 D_1.
\]
The right hand side function $g(x,y)$ is taken such that values of the function
\begin{equation}
u(x,y) = \sin(\pi x)\sin(\pi y) 
\label{eq::rhs}
\end{equation}
on the finite difference mesh 
are the entries of the exact solution of the discretized problem.

\subsubsection{Comparison of the two functionals}
\label{sect:2func}
To perform comparison of the proposed functional~\eqref{Fs} and the classical one~\eqref{Fc} four particular test linear systems are considered. 
These linear systems are obtained from test problem~\eqref{eq:model_problem} with 
the right hand side such that~\eqref{eq::rhs} is exact solution, and the following four sets of parameters:
\begin{enumerate}
\item $m = 2500$ (mesh $52\times 52$), constant coefficients $D_{1,2}$; 
\item $m = 2500$ (mesh $52\times 52$), discontinuous coefficients $D_{1,2}$;
\item $m = 10000$ (mesh $102\times 102$), constant coefficients $D_{1,2}$; 
\item $m = 10000$ (mesh $102\times 102$), discontinuous coefficients $D_{1,2}$.
\end{enumerate}
In the experiments the 2-norm is used to compute $F_s$. 
Also, we use $n = 50$ initial guess vectors, which,
taking into account stochastic convergence, 
appears to be a reasonable 
value (see also Section~\ref{sect:n}).

The optimal $\alpha$ value for the $\mathrm{RIC}_{\alpha}(0)$ 
preconditioner is sought in the interval $[0.9, 1]$,
known to contain the optimal 
value~\cite{vdVorst1989},
\cite{Henk:book}.
The number of iterations~$K$, used in the convergence functionals
$F_s$ and $F_c$, is determined such that a required tolerance is achieved
for a reasonable (not yet optimized) value of $\alpha$.
In the experiments below we set the required tolerance for the residual
norm reduction $\|r_k\|/\|r_0\|$ to $10^{-7}$, where the residual is defined as $r_k = b - Ax_k$.
This tolerance value yields the values of $K$ given in Table~\ref{tab::K}.
To find the optimal parameters $\alpha^*_c$ and $\alpha^*_s$, corresponding to the classical and to the stochastic convergence functionals, 
respectively, Brent's optimization method is applied. 
The accuracy of the optimization procedure is set to $10^{-5}$
which is sufficient for our purposes.
The computed optimal values $\alpha^*_c$ and $\alpha^*_s$ are 
given in Table~\ref{tab::K}.
Here insignificant (taking into account 
optimization accuracy) digits are shown within brackets.


\begin{table}[!h]
\centering
\caption{Number of iterations $K$ used in the convergence functionals
$F_s$ and $F_c$ (for the tolerance is $10^{-7}$) and corresponding optimal parameters $\alpha^*_s$ and $\alpha^*_c$}
\begin{tabular}{lccc}
\toprule
Test case & $K$ & $\alpha^*_s$ & $\alpha^*_c$\\
\midrule
$m=2500$, constant $D_{1,2}$ & 20 & $0.98257(07)$ & $ 0.99618(02)$ \\
$m=2500$, discontinuous $D_{1,2}$ & 30 & $0.97671(44)$ & $0.99999(47)$ \\
$m=10000$, constant $D_{1,2}$ & 35 & $0.99245(52)$ & $0.99900(93)$  \\
$m=10000$, discontinuous $D_{1,2}$ & 45 & $0.99451(65)$ & $0.99999(33)$ \\
\bottomrule
\end{tabular}
\label{tab::K}
\end{table}

Figures~\ref{fig::52_nonsmooth}--\ref{fig::102_smooth} show dependence of both functionals on the preconditioner parameter $\alpha$. 
Plots (a) correspond to stochastic functional and plots (b) correspond to the classical one. 
As can be seen in the plots,
the stochastic functional $F_s$ tends, more often than $F_c$, to have a 
minimum \emph{close} to one, rather than \emph{exactly} at one.
%
To investigate this difference, we plot eigenvalue distribution of the preconditioned matrices corresponding to $\alpha^*_s$ and $\alpha^*_c$ and compare eigenvalue clustering.
Figures~\ref{fig:stoch_class_52_nonsmooth_eigvals},~\ref{fig:stoch_class_52_smooth_eigvals},~\ref{fig:stoch_class_102_nonsmooth_eigvals} and~\ref{fig:stoch_class_102_smooth_eigvals} show the spectra of the preconditioned matrices corresponding to $\alpha^*_s$ and $\alpha^*_c$.
Spectrum distribution plots demonstrate that $\alpha^*_s$ yields spectra
better sparsified at their lower part than $\alpha^*_c$.
Consequently, 
convergence of the preconditioned CG, which is tested on  problem~\eqref{eq:model_problem}, for $\alpha^*_s$ is faster than for $\alpha^*_c$ , see Figure~\ref{fig:conv_compare_52_nonsmooth},~\ref{fig:conv_compare_52_smooth},~\ref{fig:conv_compare_102_nonsmooth} and~\ref{fig:conv_compare_102_smooth}.

\subsubsection{Other possible distributions of $x_0$}
\label{sec::gf}

To see how our approach depends on the choice of the distribution of 
$x_* - x_0$, in this section we test our optimization procedure for 
$x_* - x_0$ whose entries are taken from a stationary Gaussian
random field.  More specifically, 
assume $x_* - x_0$ is a discretization of a 
stationary Gaussian random field with zero mean and the Gaussian covariance function~\cite{kroese2015spatial,dietrich1997fast} 
\[
C(x, y) = \gamma(x - y) = \exp\left(\frac{\|x - y\|_2^2}{\sigma^2} \right),
\]
where $\sigma^2$ is a parameter of the random field.
We consider the test problem~\eqref{eq:model_problem}, $m=2500$ and the discontinuous coefficients~$D_{1,2}$.
To observe the effect of this distribution of $x_* - x_0$ on the optimal parameter $\alpha^*_s$, we solve the same optimization problem as in the previous section, but use the Gaussian random field with different values of $\sigma^2 \in \{ 10^{-1}, 10^{-2}, 10^{-3}, 10^{-4}\}$ to generate trial vectors for computation of $F_s$.
The obtained optimal parameters $\alpha^*_s$ are practically identical for all the considered values of $\sigma^2$ and almost indistinguishable from the 
parameter values of the normal distribution case, see Table~\ref{tab::K}.
These values of $\alpha^*_s$ also yield practically indistinguishable residual norm convergence plots, therefore we present in Figure~\ref{fig::x0_norm_gf_52_nonsmooth} only plots for the standard normal distribution of $x_* - x_0$ and for the Gaussian random field with $\sigma^2 = 0.01$.
Thus, the proposed approach can be used not only for the trial vectors generated from the standard normal distribution but also from the stationary Gaussian random field.

\begin{figure}[!h]
    \centering
    \begin{subfigure}[t]{0.45\textwidth}
    \centering
 \includegraphics[width=\textwidth]{{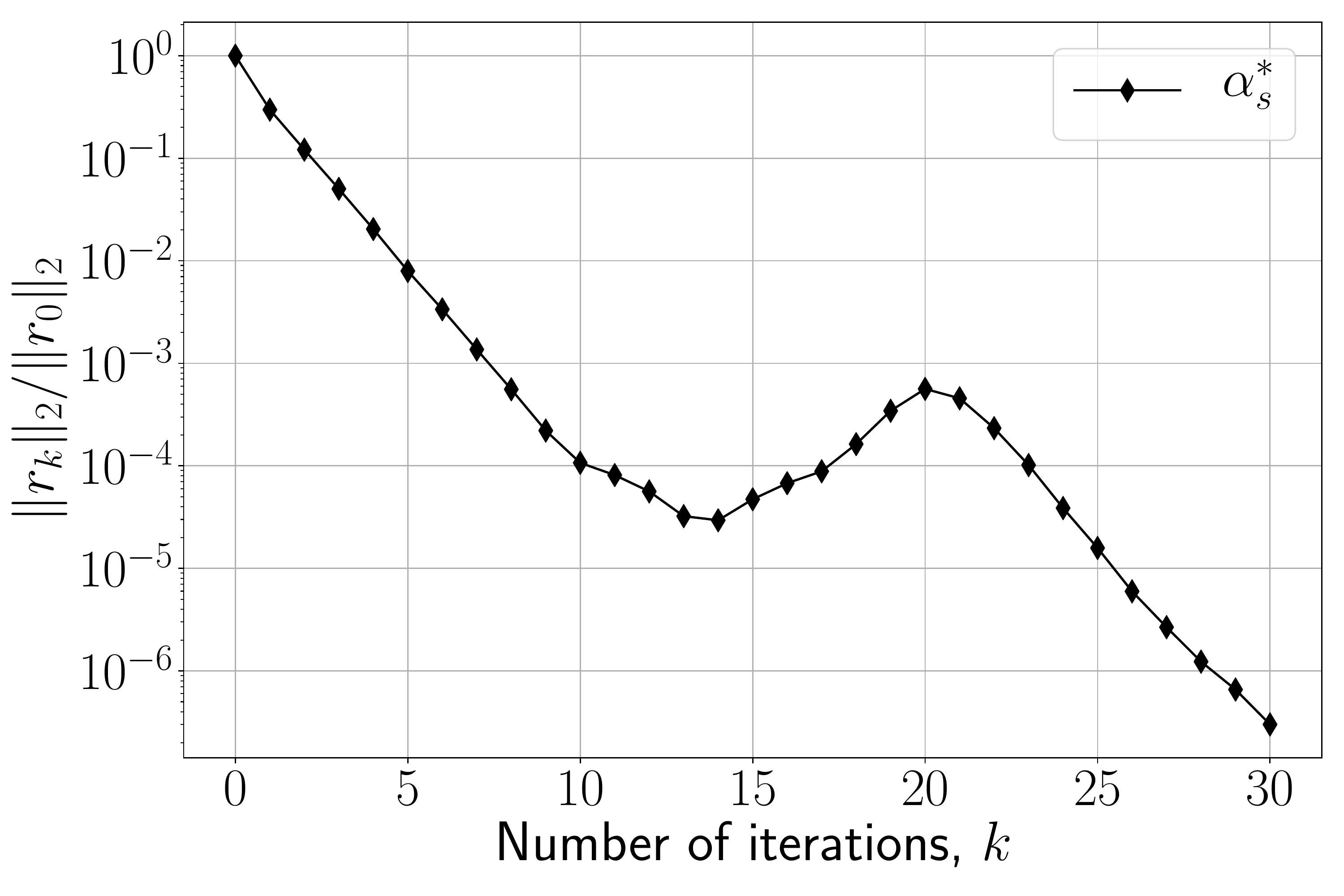}}
    \caption{}
\label{fig:x0_norm_stoch_52_nonsmooth_iter}
\end{subfigure}
~
	\begin{subfigure}[t]{0.45\textwidth}
    \centering
\includegraphics[width=\textwidth]{{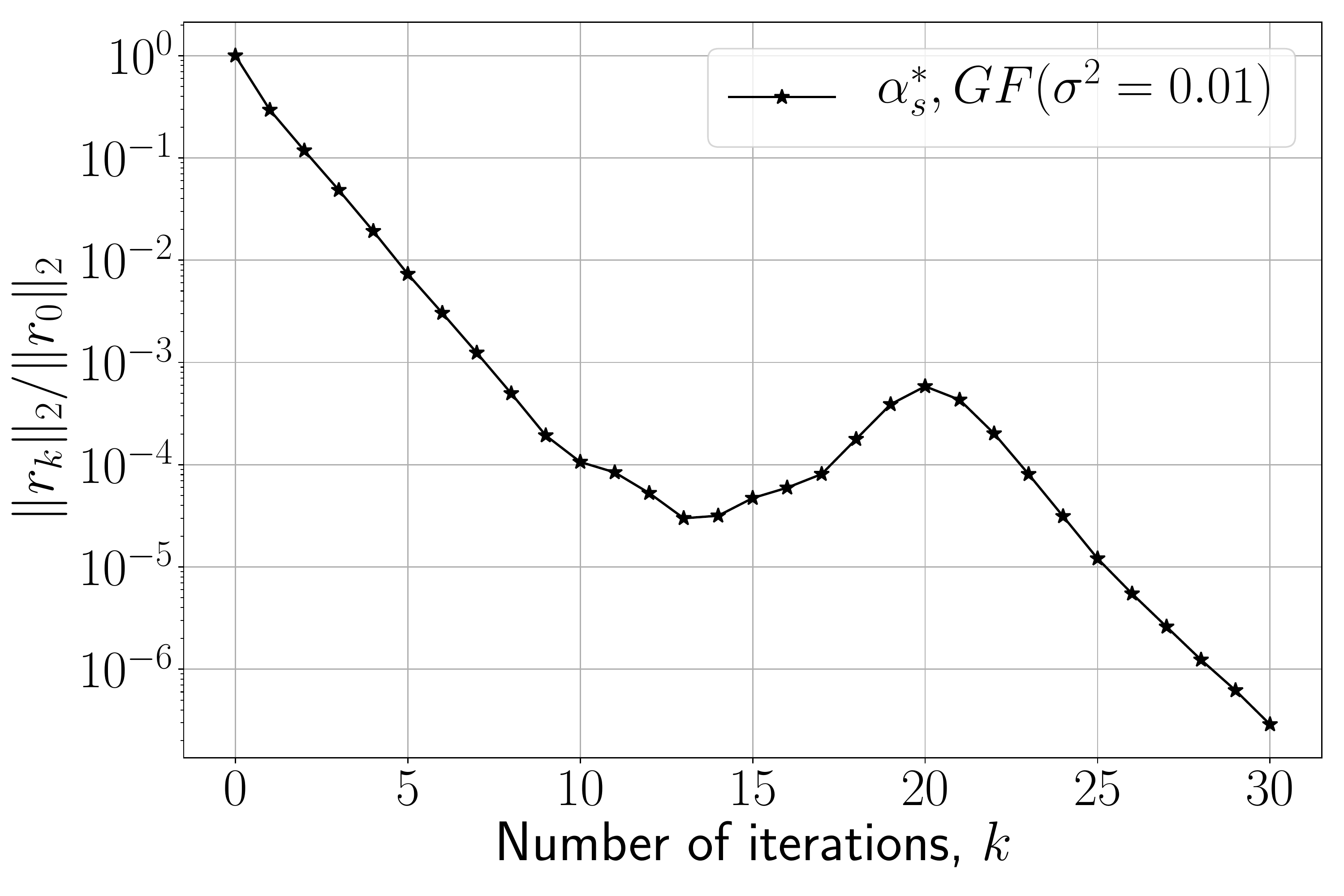}}
    \caption{}
    \label{fig:x0_gf_stoch_52_nonsmooth_iter}
    \end{subfigure}
\caption{Comparison of convergence for the optimal parameters obtained with trial vectors generated from standard normal distribution (left plot) and Gaussian random field with $\sigma^2 = 0.01$ (right plot)}
\label{fig::x0_norm_gf_52_nonsmooth}
\end{figure}
    
\subsubsection{The number of random initial guess vectors $n$}
\label{sect:n}
The costs of optimization procedure depend heavily on the choice of the number of initial guess vectors $n$, see~ Section~\ref{sec::opt_cost}. 
As mentioned above, we use $n=50$ in all the experiments presented above.
In this section we test how sensitive the obtained results are to the choice of~$n$.
It appears that similar results can be obtained with smaller values of~$n$.
We consider the test case with $m=2500$ and constant coefficients $D_{1,2}$.
According to Table~\ref{tab::K}, number of iterations $K$ for this test case is 
set to~$20$.
For this experiment setting, the dependence of the proposed stochastic convergence functional $F_s$ on the preconditioner parameter $\alpha$ for several values of $n$
is plot in Figure~\ref{fig::n}.
As we see in the plots, the larger $n$, the smoother the
dependence line and already $n=10$ is enough for an adequate 
representation of the considered dependence.
In Figure~\ref{fig::conf_int} the confidence interval is plot for $F_s$.



\begin{figure}[!h]
    \centering
    \begin{subfigure}[t]{0.45\textwidth}
    \centering
 \includegraphics[width=\textwidth]{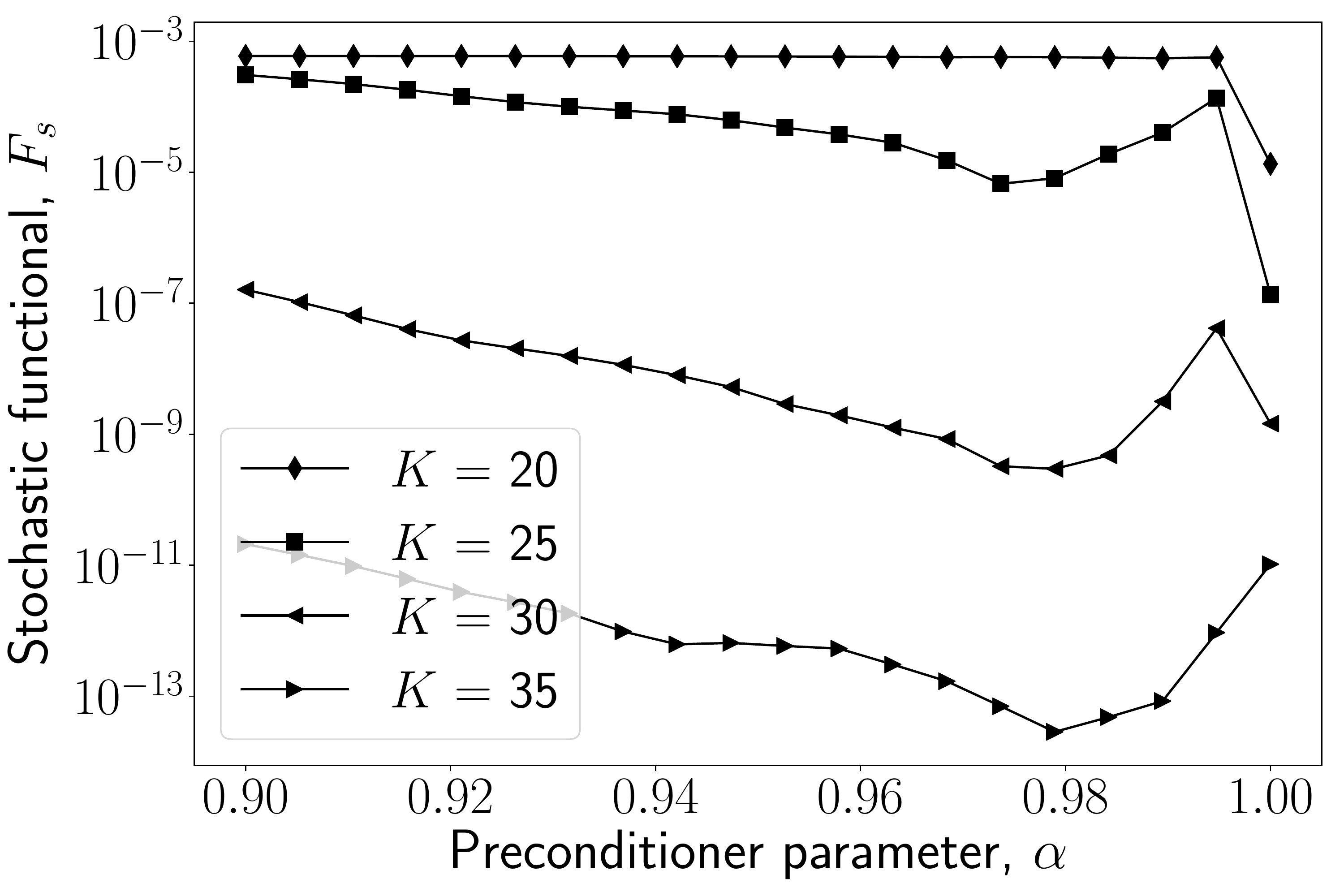}
    \caption{}
\label{fig:stoch_52_nonsmooth_iter}
\end{subfigure}
~
	\begin{subfigure}[t]{0.45\textwidth}
    \centering
\includegraphics[width=\textwidth]{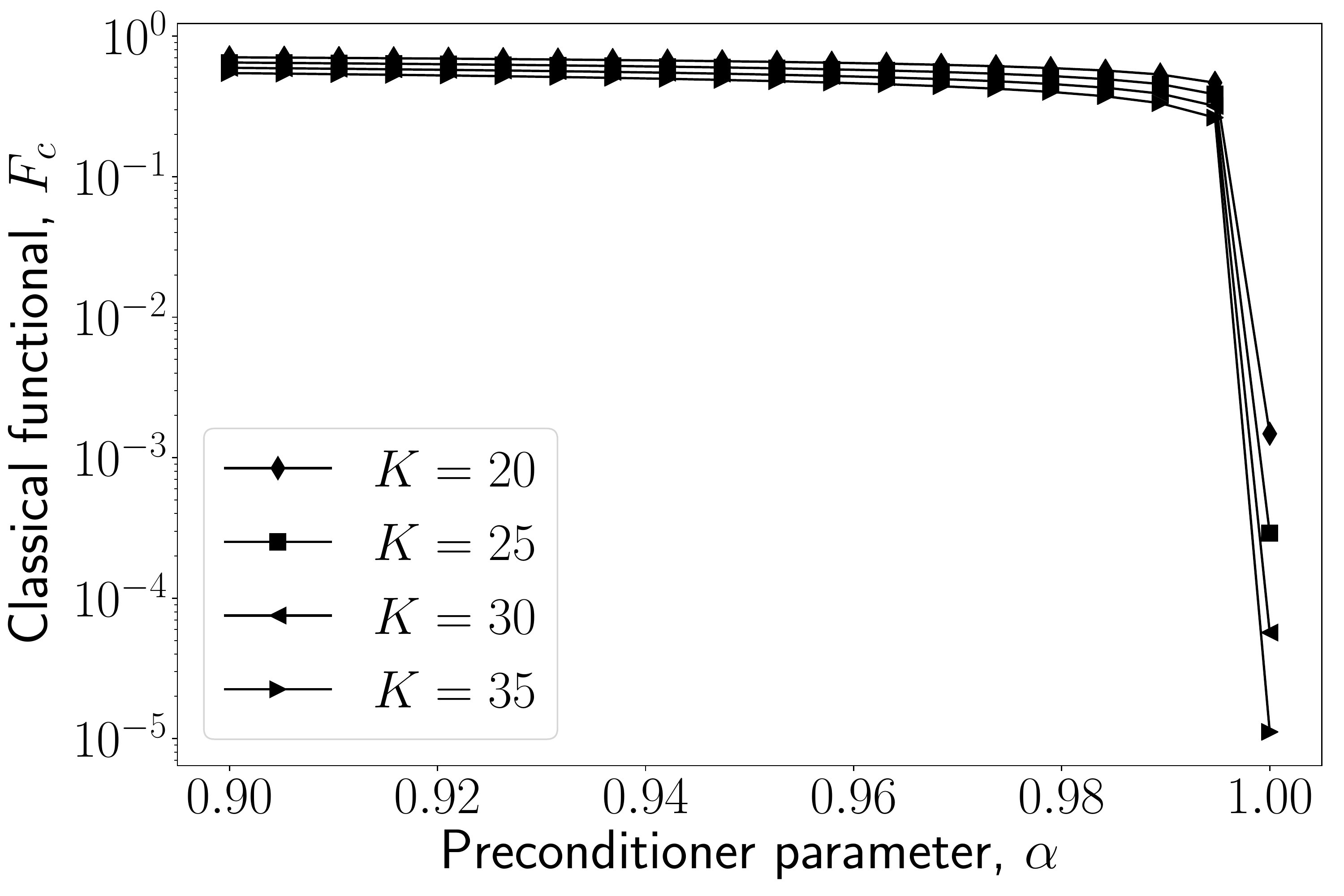}
    \caption{}
    \label{fig:class_52_nonsmooth_iter}
    \end{subfigure}
\\
\begin{subfigure}[t]{0.47\textwidth}
\centering
\includegraphics[width=\textwidth,height=0.24\textheight]{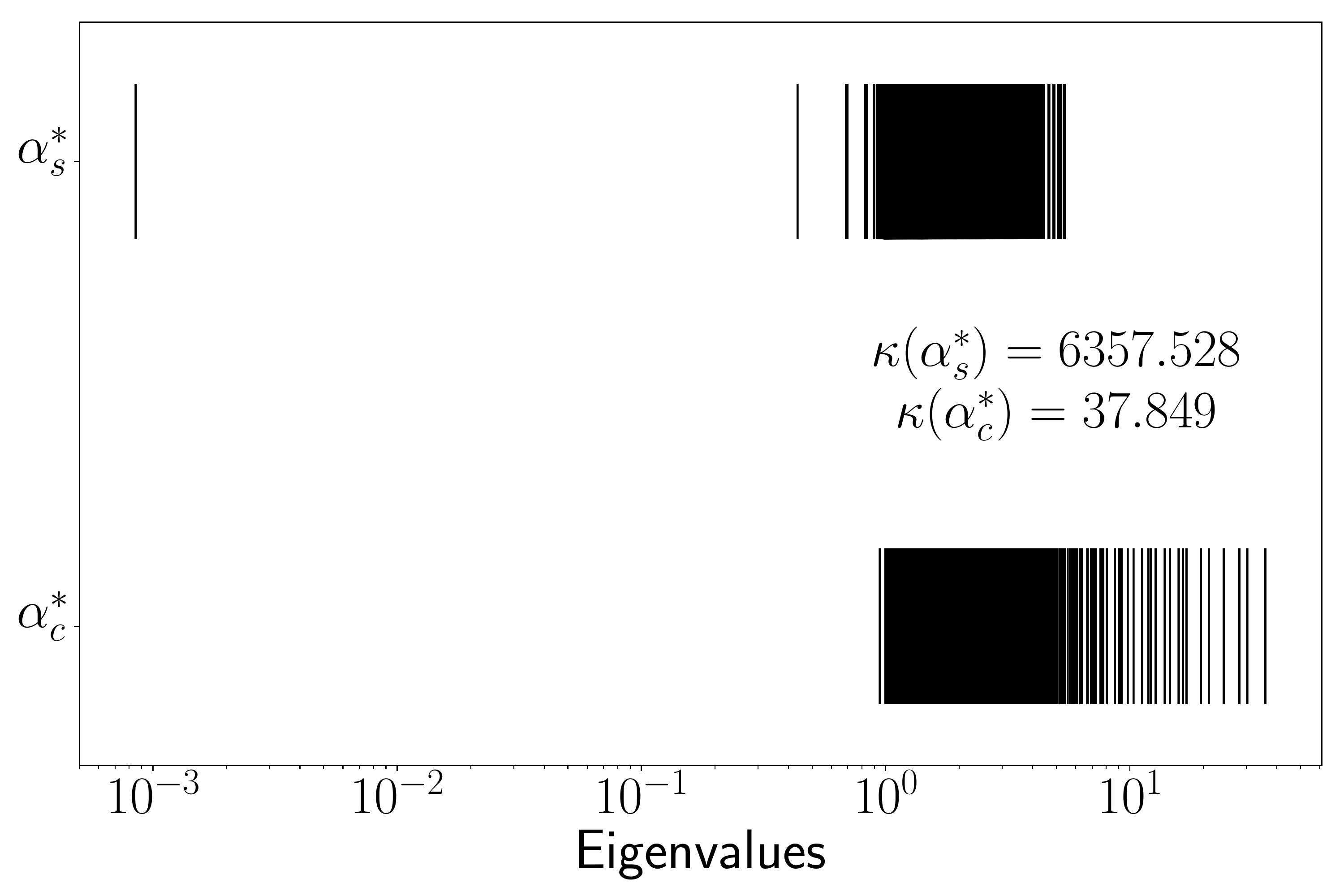}
    \caption{}
 \label{fig:stoch_class_52_nonsmooth_eigvals}
\end{subfigure}
~
    \begin{subfigure}[t]{0.5\textwidth}
    \centering
 \includegraphics[width=\textwidth]{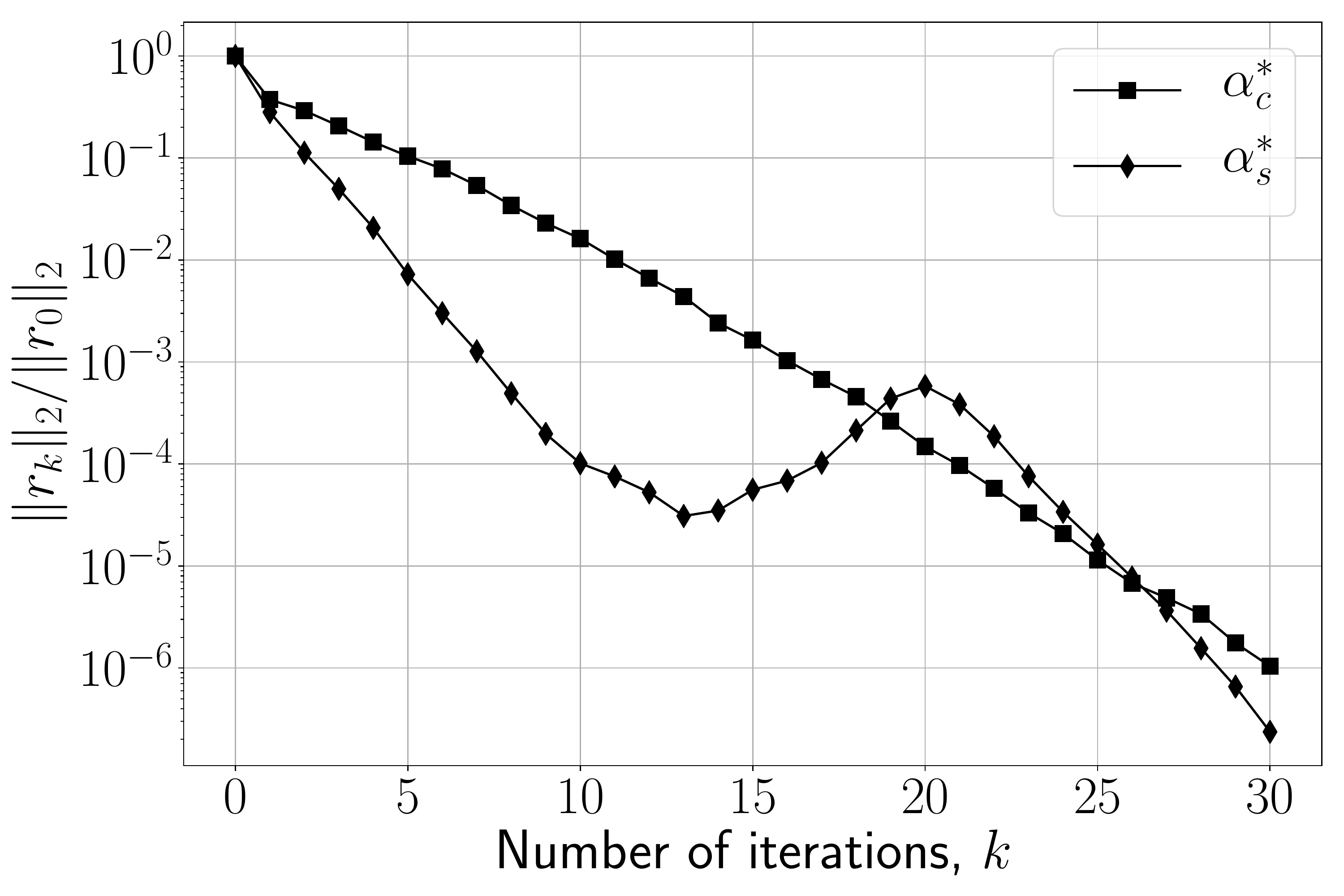}
    \caption{}
\label{fig:conv_compare_52_nonsmooth}
    \end{subfigure}
    \caption{Test case $m = 2500$, discontinuous coefficients $D_{1,2}$.
    Plots~(a),(b): dependence of the functionals $F_s$, $F_c$ on the preconditioner parameter $\alpha$ for different iteration number $K$.
    Plots~(c): eigenvalues of the preconditioned matrix for the optimal $\alpha_s^*$ (based on $F_s$) and $\alpha_c^*$ (based on $F_c$).
    Plot~(d): Residual norm convergence of CG preconditioned by $\mathrm{RIC}_{\alpha}(0)$ for $\alpha_s^*$ and $\alpha_c^*$.}
\label{fig::52_nonsmooth}
\end{figure}

\begin{figure}[!h]
\centering
\begin{subfigure}[t]{0.45\textwidth}
\centering
\includegraphics[width=\textwidth]{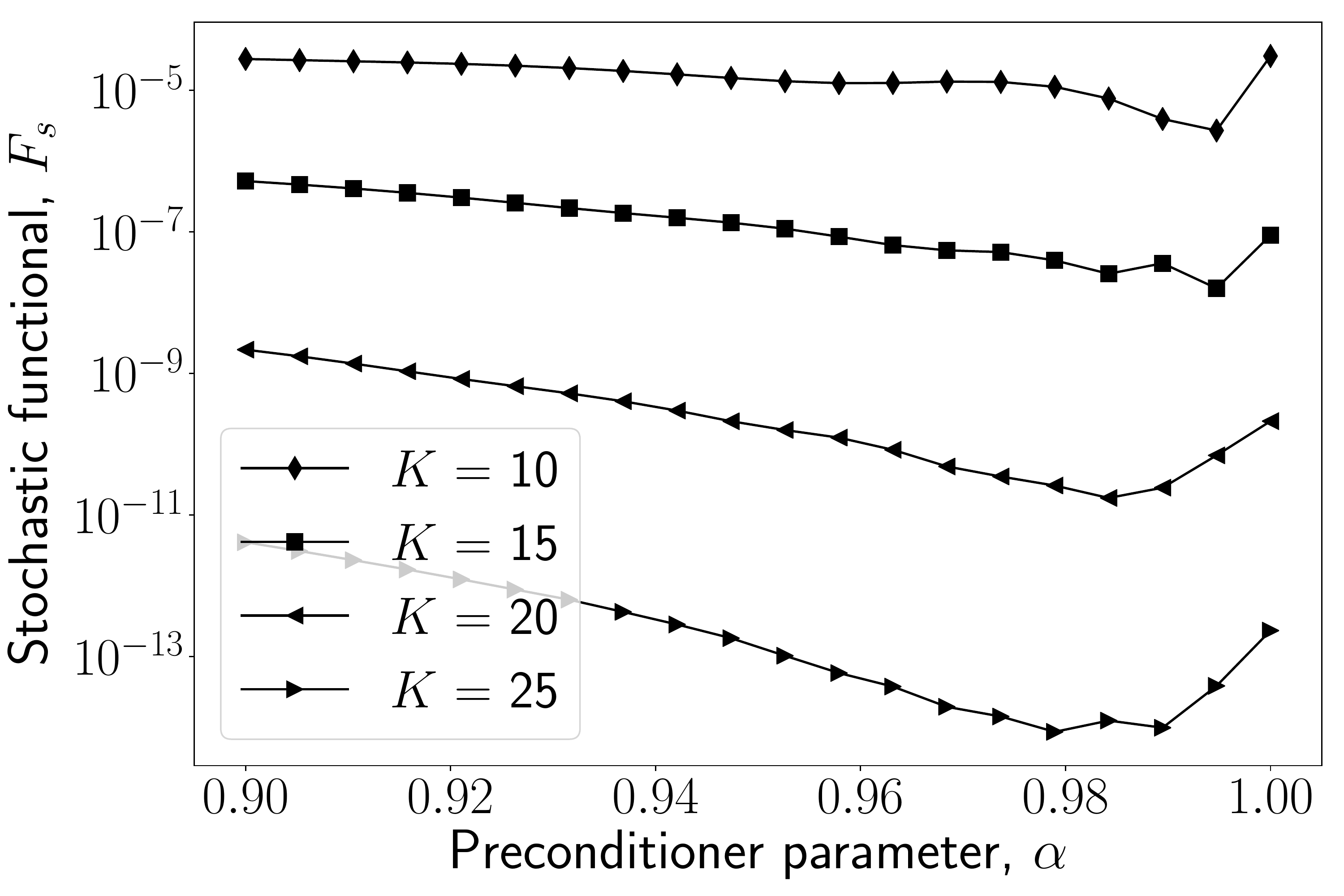}
\caption{}
\label{fig:stoch_52_smooth_iter}
\end{subfigure}
~
\begin{subfigure}[t]{0.45\textwidth}
\centering
\includegraphics[width=\textwidth]{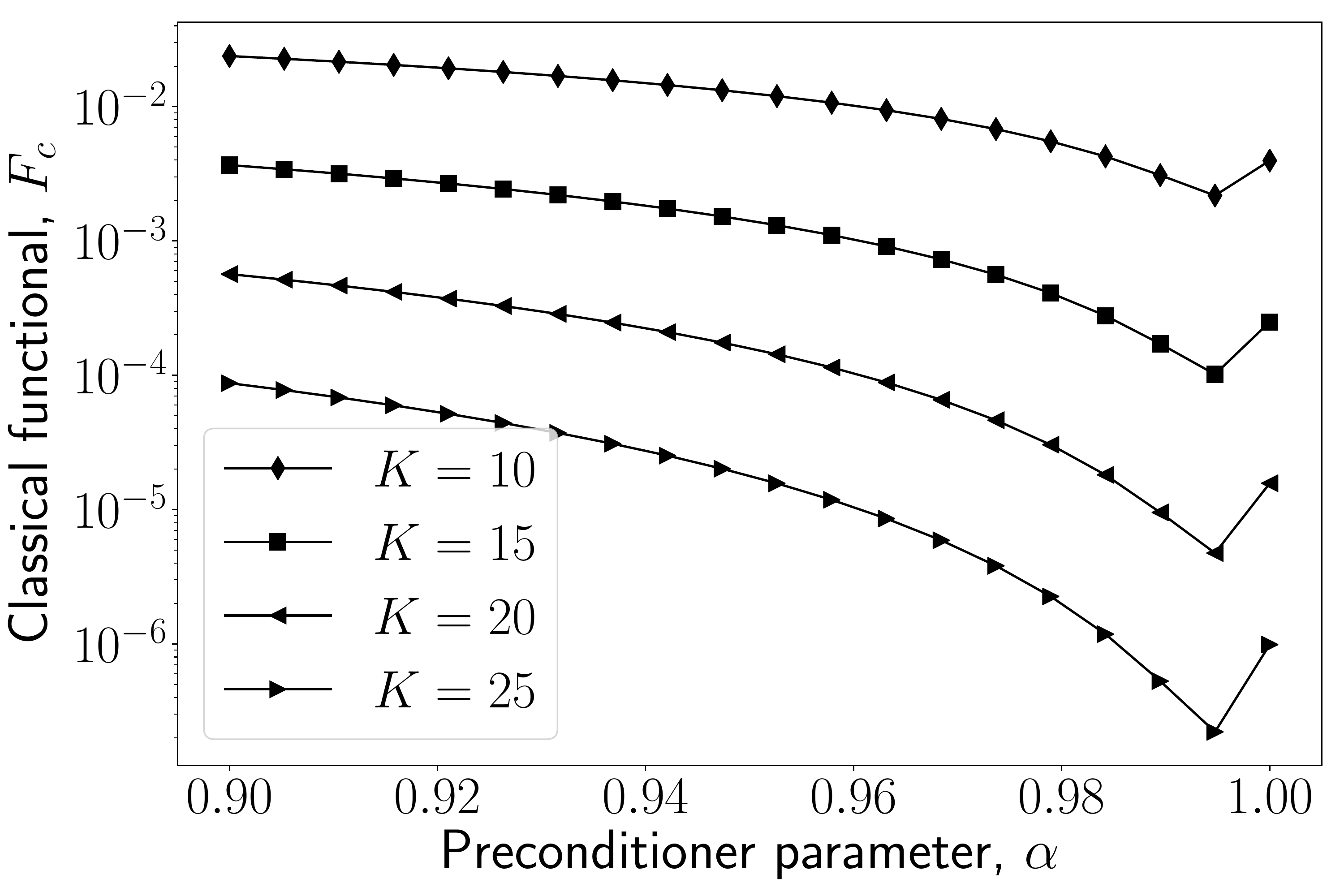}
\caption{}
\label{fig:class_52_smooth_iter}
\end{subfigure}
\\
\begin{subfigure}[t]{0.45\textwidth}
\centering
\includegraphics[width=\textwidth,height=0.24\textheight]{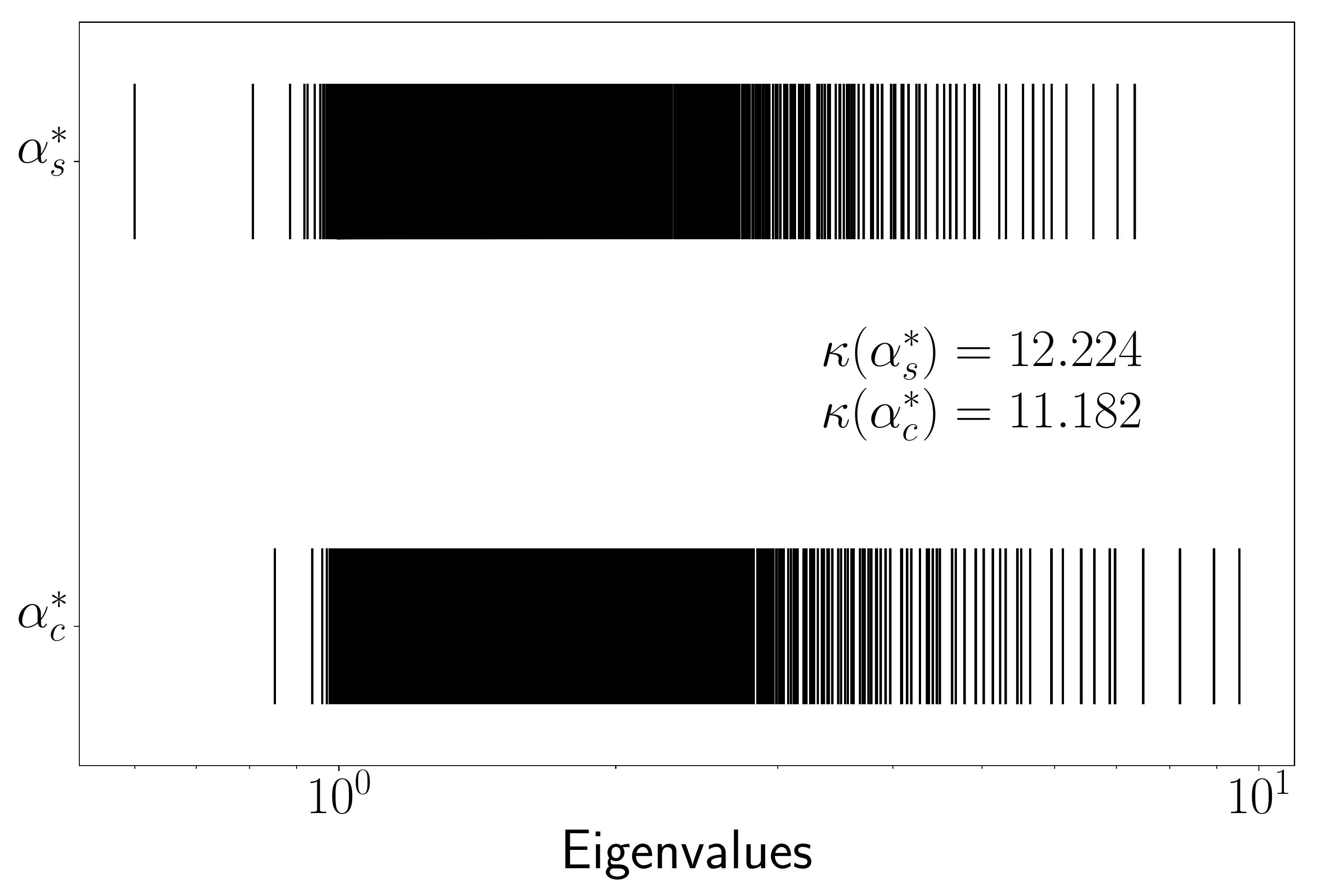}
\caption{}
\label{fig:stoch_class_52_smooth_eigvals}
\end{subfigure}
~
\begin{subfigure}[t]{0.5\textwidth}
\centering
\includegraphics[width=\textwidth]{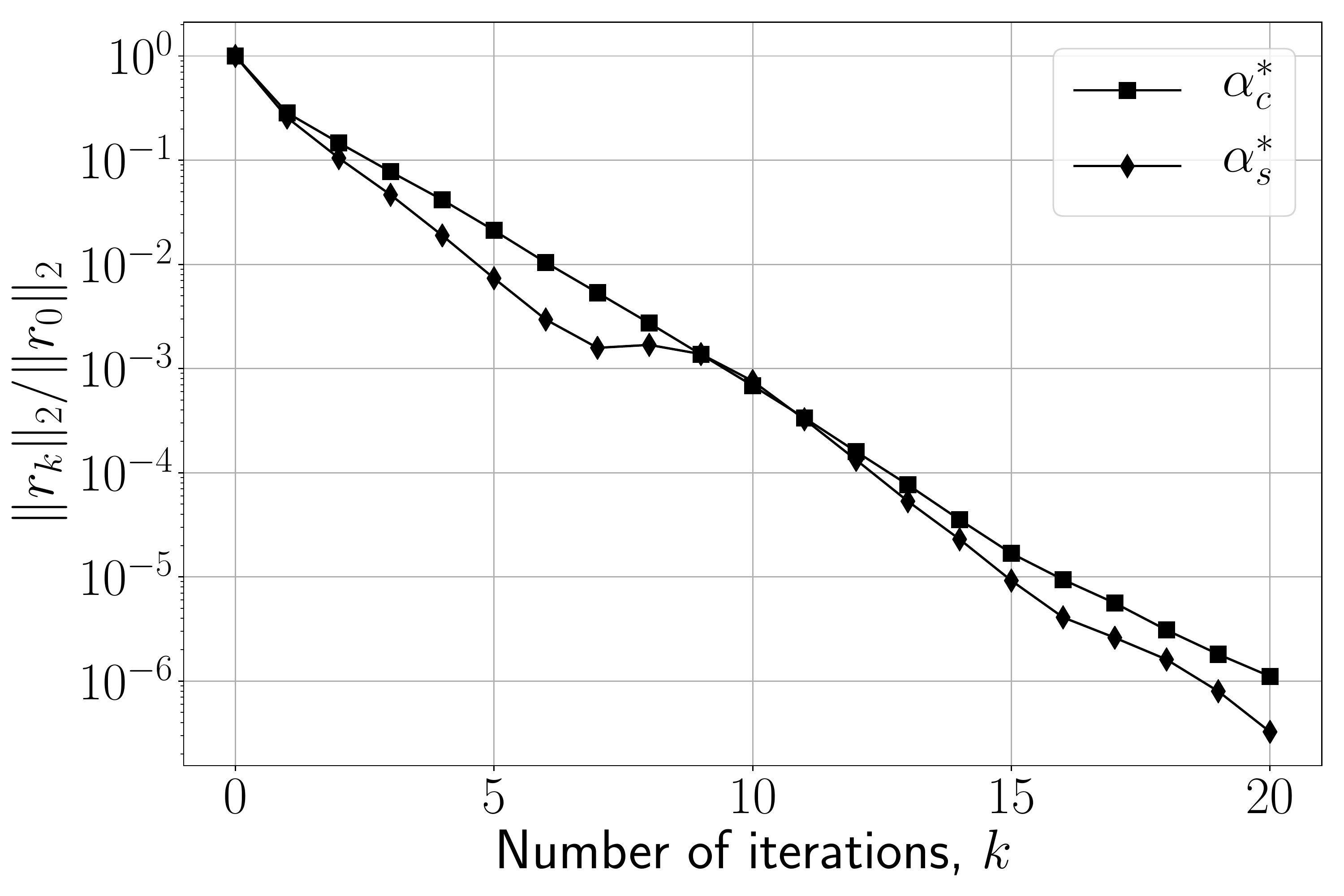}
\caption{}
\label{fig:conv_compare_52_smooth}
\end{subfigure}
    
    \caption{Test case $m = 2500$, constant coefficients $D_{1,2}$.
    Plots~(a),(b): dependence of the functionals $F_s$, $F_c$ on the preconditioner parameter $\alpha$ for different iteration number $K$.
    Plot~(c): eigenvalues of the preconditioned matrix for the optimal $\alpha_s^*$ (based on~$F_s$) and $\alpha_c^*$ (based on~$F_c$).
    Plot~(d): Residual norm convergence of CG preconditioned by $\mathrm{RIC}_{\alpha}(0)$ for $\alpha_s^*$ and $\alpha_c^*$.}
    \label{fig::52_smooth}
\end{figure}

\begin{figure}[!h]
\centering
\begin{subfigure}[t]{0.45\textwidth}
\centering
\includegraphics[width=\textwidth]{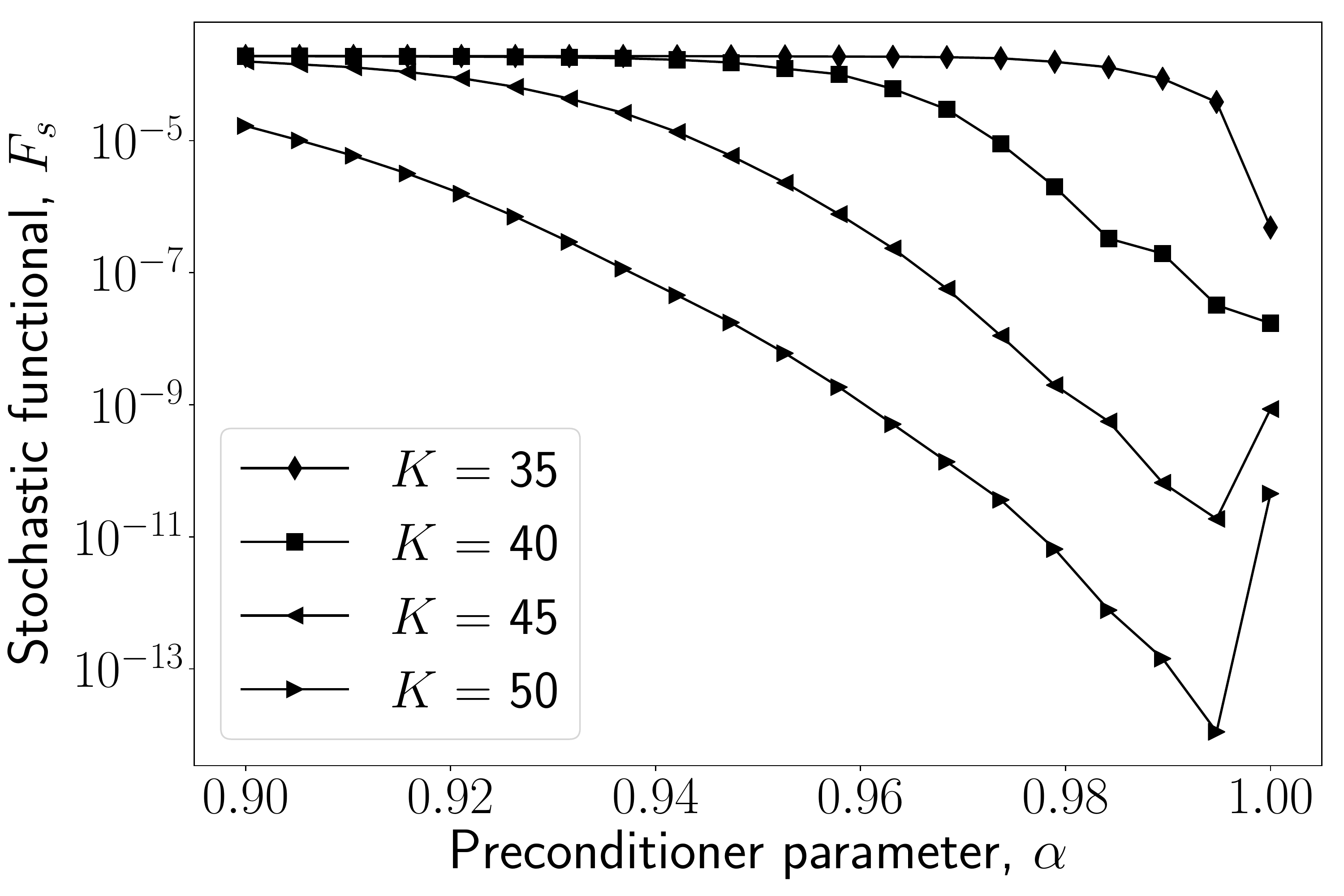}
\caption{}
\label{fig:stoch_102_nonsmooth_iter}
\end{subfigure}
~
\begin{subfigure}[t]{0.45\textwidth}
\centering
\includegraphics[width=\textwidth]{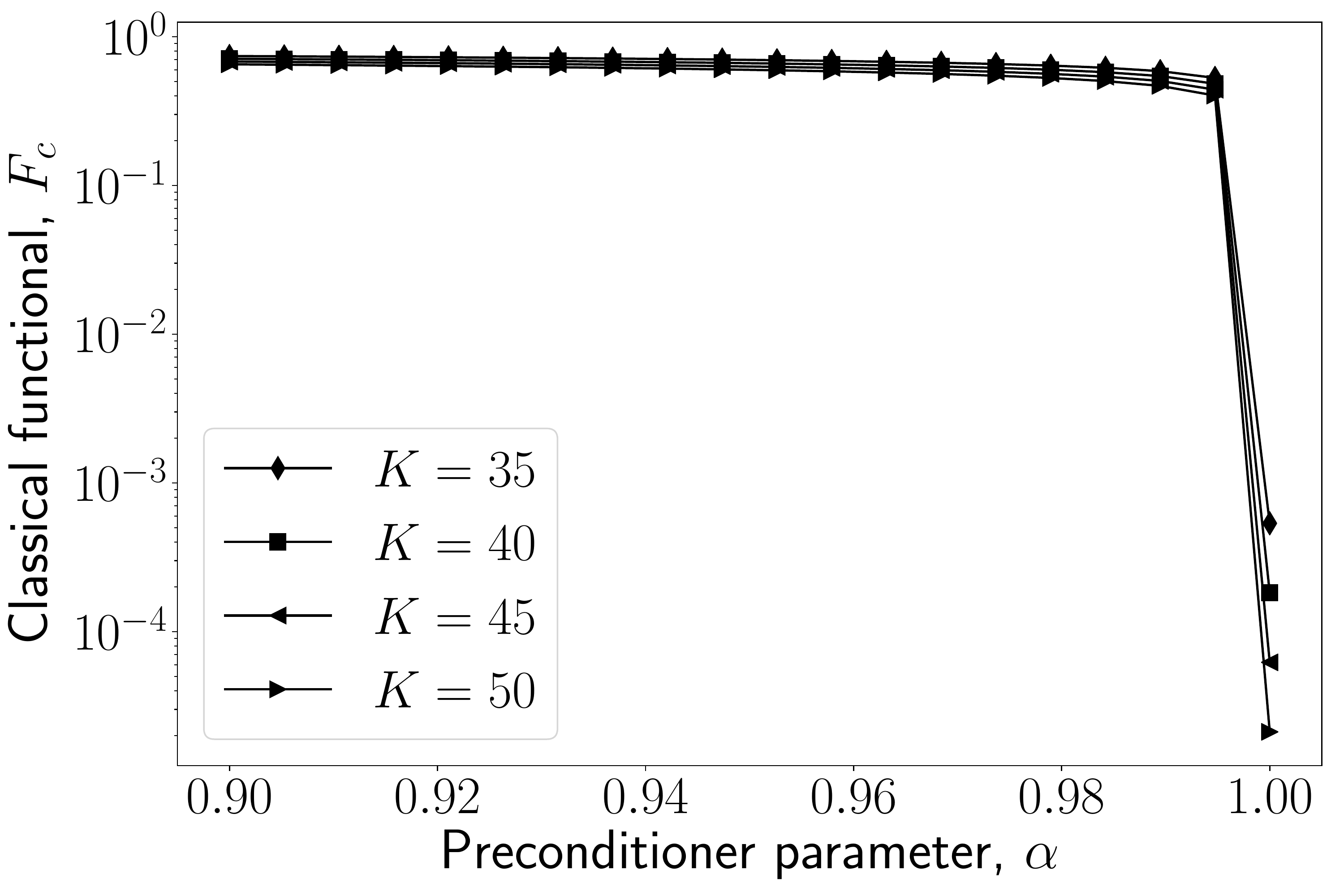}
\caption{}
\label{fig:class_102_nonsmooth_iter}
\end{subfigure}
\\
\begin{subfigure}[t]{0.45\textwidth}
\centering
\includegraphics[width=\textwidth,height=0.24\textheight]{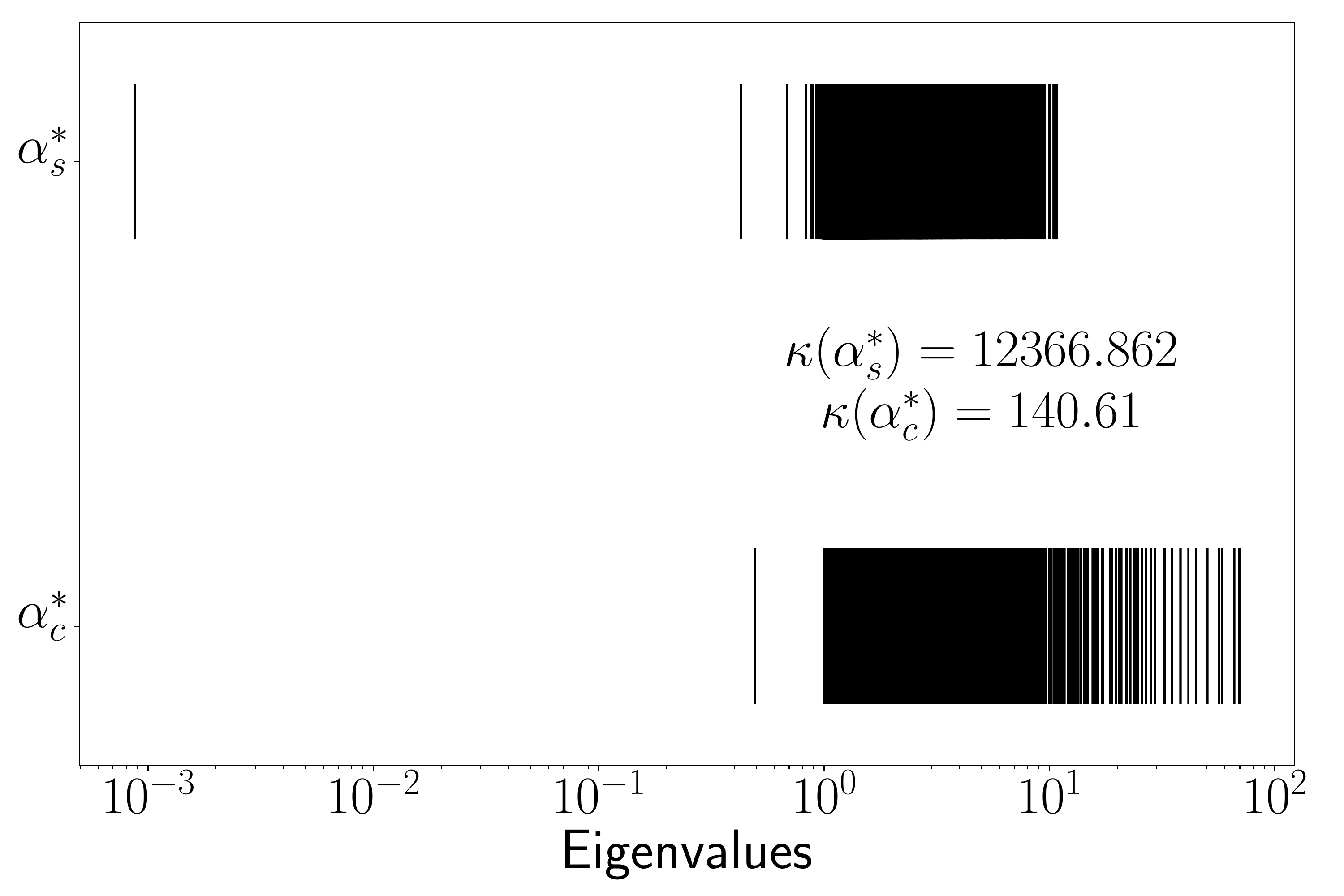}
\caption{}
\label{fig:stoch_class_102_nonsmooth_eigvals}
\end{subfigure}
~
\begin{subfigure}[t]{0.5\textwidth}
\centering
\includegraphics[width=\textwidth]{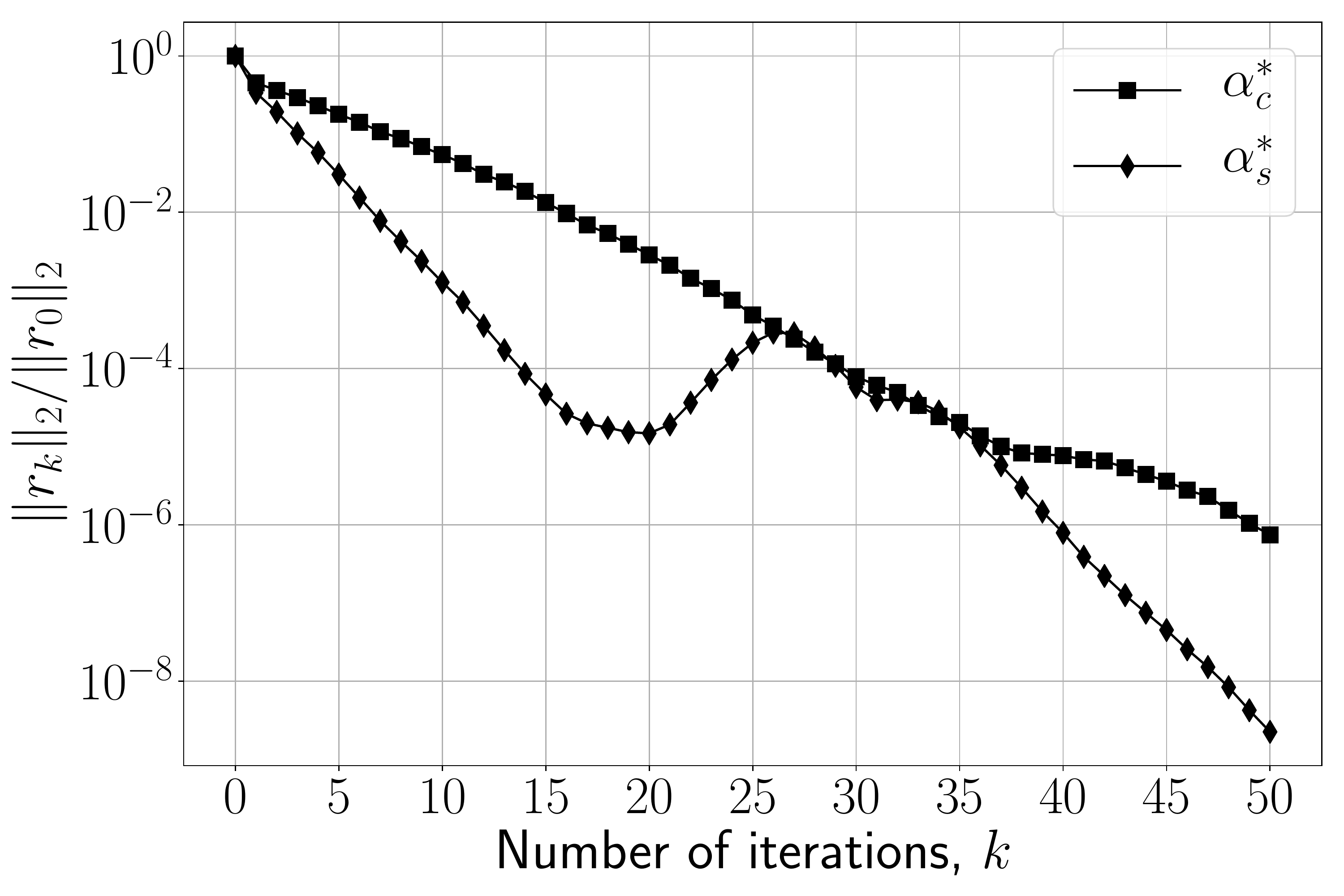}
\caption{}
\label{fig:conv_compare_102_nonsmooth}
\end{subfigure}

\caption{Test case $m = 10000$, discontinuous coefficients $D_{1,2}$.
    Plots~(a),(c): dependence of the functionals $F_s$, $F_c$ on the 
    preconditioner parameter $\alpha$ for different iteration number $K$.
    Plots~(c): eigenvalues of the preconditioned matrix for the 
    optimal $\alpha_s^*$ (based on~$F_s$) and $\alpha_c^*$ 
    (based on~$F_c$).
    Plot~(d): Residual norm convergence of CG preconditioned by $\mathrm{RIC}_{\alpha}(0)$ for $\alpha_s^*$ and $\alpha_c^*$.}
    \label{fig::102_nonsmooth}
\end{figure}

\begin{figure}[!h]
\centering
\begin{subfigure}[t]{0.45\textwidth}
\centering
\includegraphics[width=\textwidth]{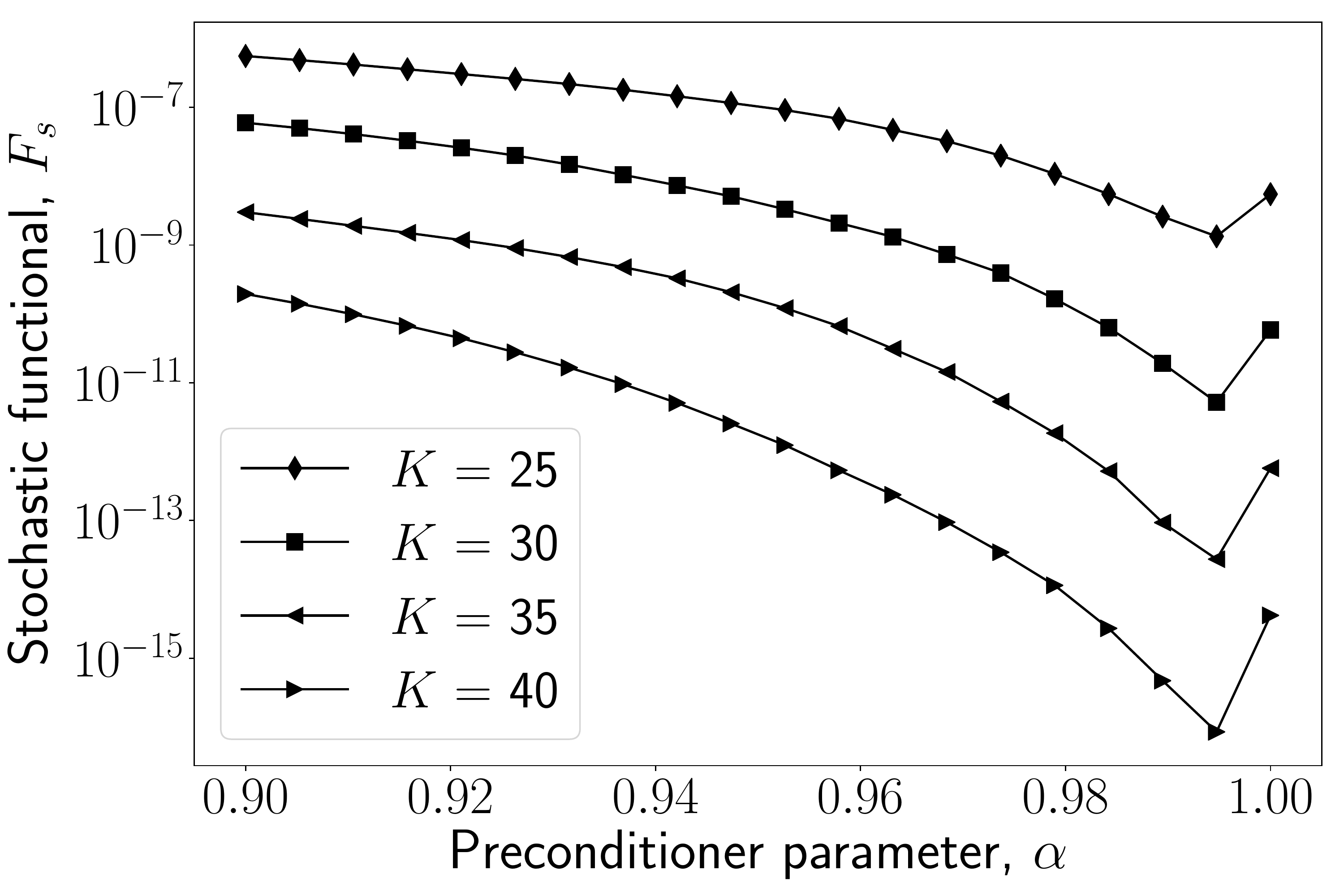}
\caption{}
\label{fig:stoch_102_smooth_iter}
\end{subfigure}
~
\begin{subfigure}[t]{0.45\textwidth}
\centering
\includegraphics[width=\textwidth]{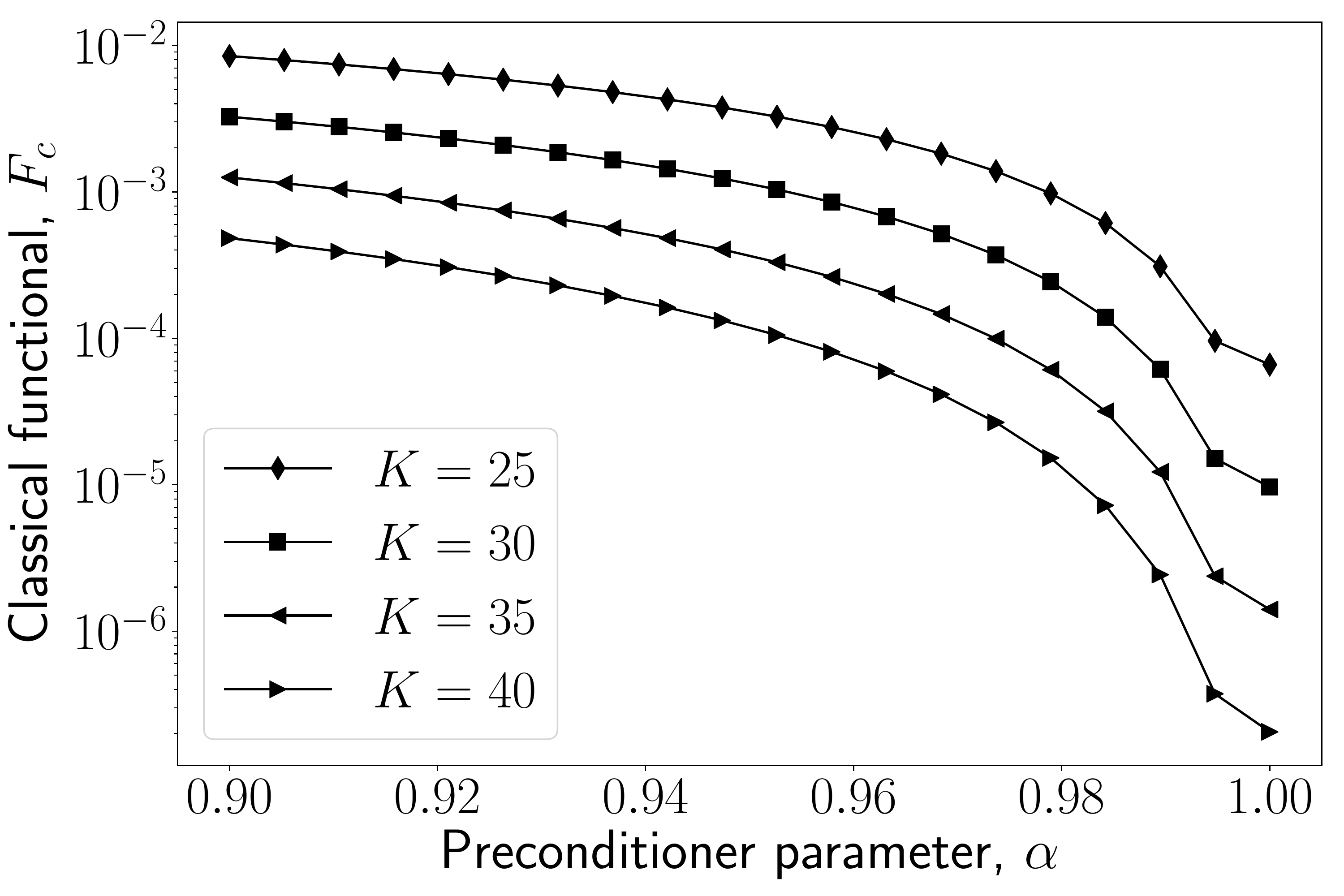}
\caption{}
\label{fig:class_102_smooth_iter}
\end{subfigure}
\\
\begin{subfigure}[t]{0.45\textwidth}
\centering
\includegraphics[width=\textwidth,height=0.24\textheight]{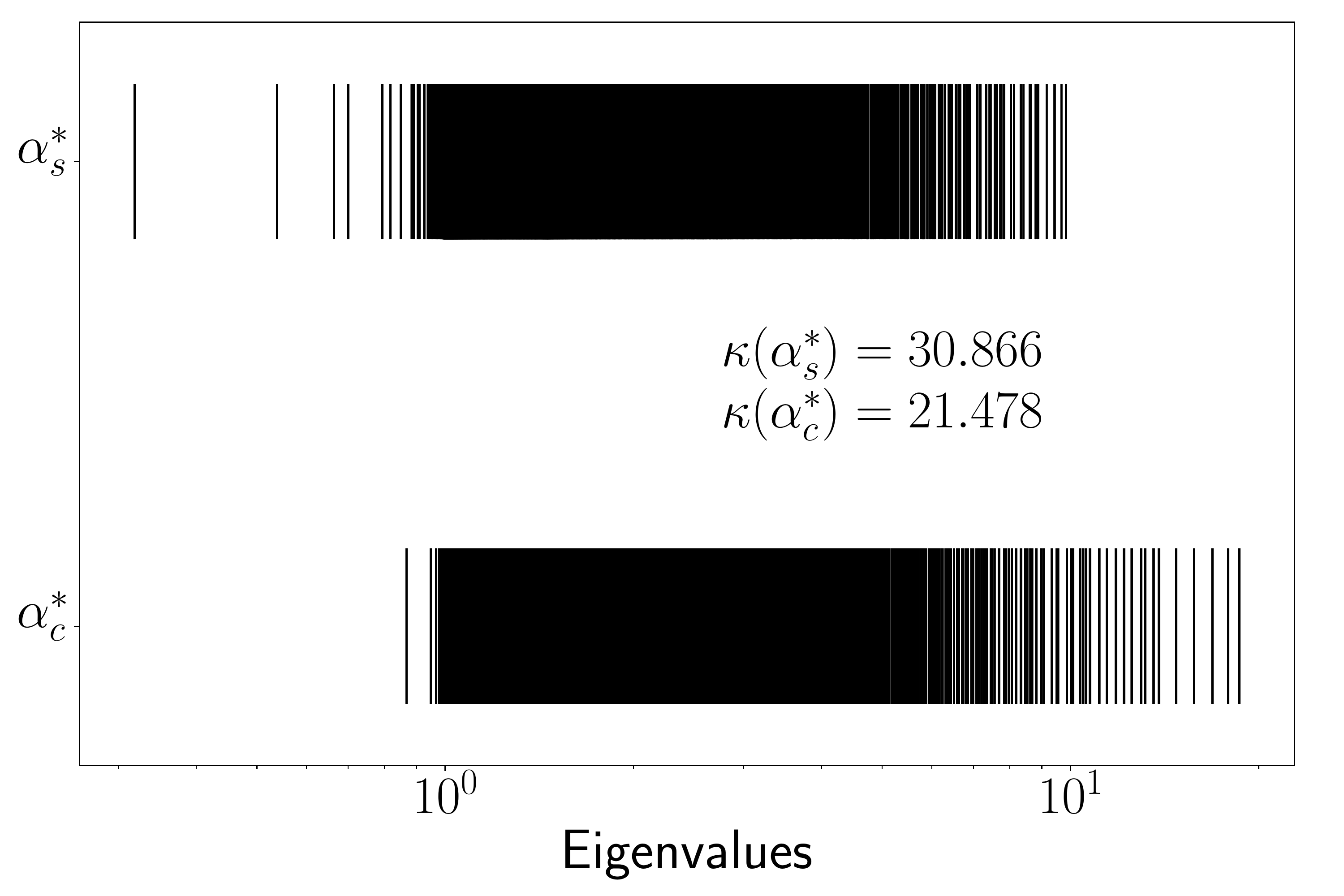}
\caption{}
\label{fig:stoch_class_102_smooth_eigvals}
\end{subfigure}
~
\begin{subfigure}[t]{0.5\textwidth}
\centering
\includegraphics[width=\textwidth]{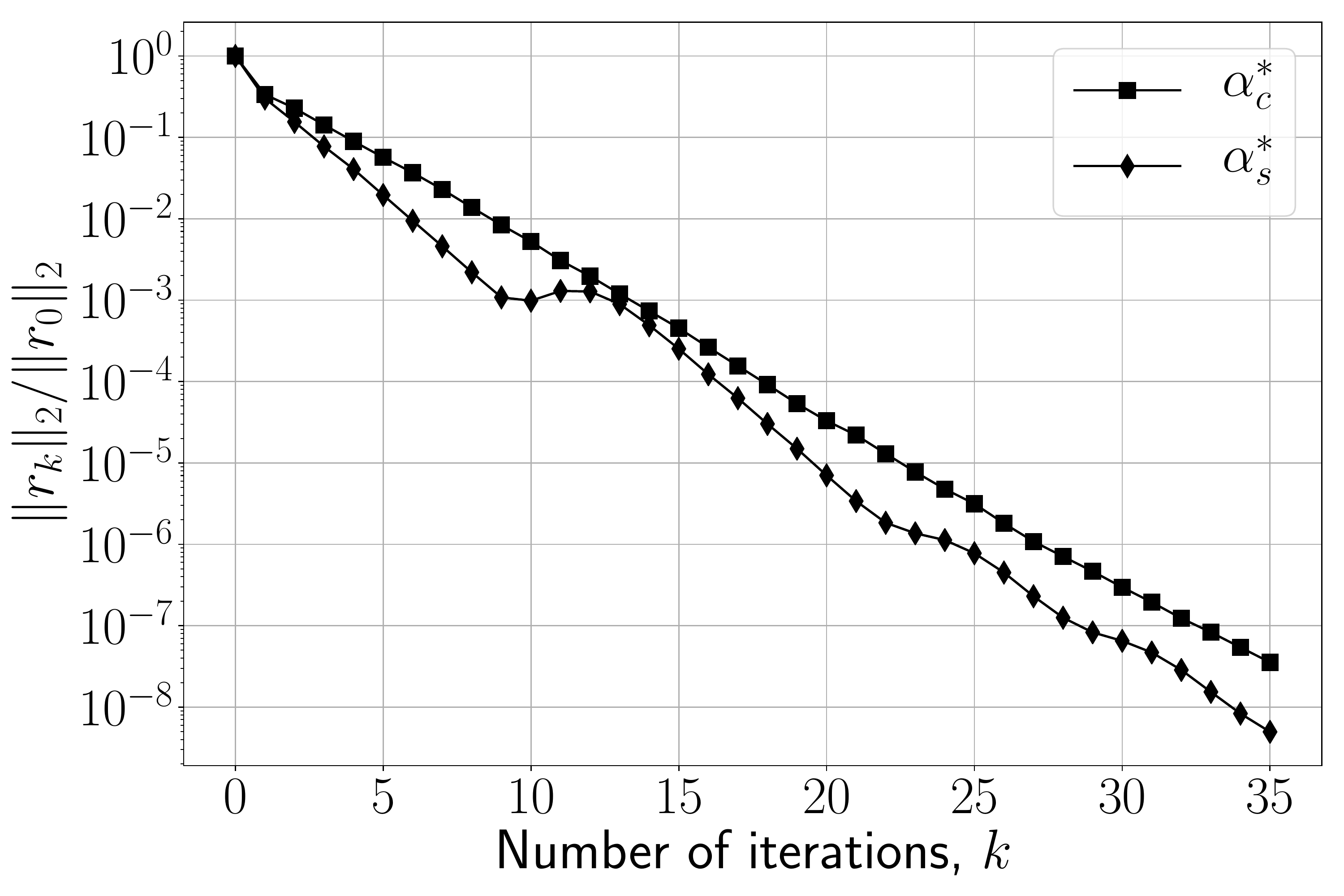}
\caption{}
\label{fig:conv_compare_102_smooth}
\end{subfigure}

    \caption{Test case $m = 10000$, constant coefficients $D_{1,2}$.
    Plots~(a),(b): dependence of the functionals $F_s$, $F_c$ on the 
    preconditioner parameter $\alpha$ for different iteration number $K$.
    Plots~(c): eigenvalues of the preconditioned matrix for the 
    optimal $\alpha_s^*$ (based on~$F_s$) and $\alpha_c^*$ 
    (based on~$F_c$).
    Plot~(d): Residual norm convergence of CG preconditioned by $\mathrm{RIC}_{\alpha}(0)$ for $\alpha_s^*$ and $\alpha_c^*$.}
    \label{fig::102_smooth}
\end{figure}

\begin{figure}
\centering
\includegraphics[width=0.7\textwidth]{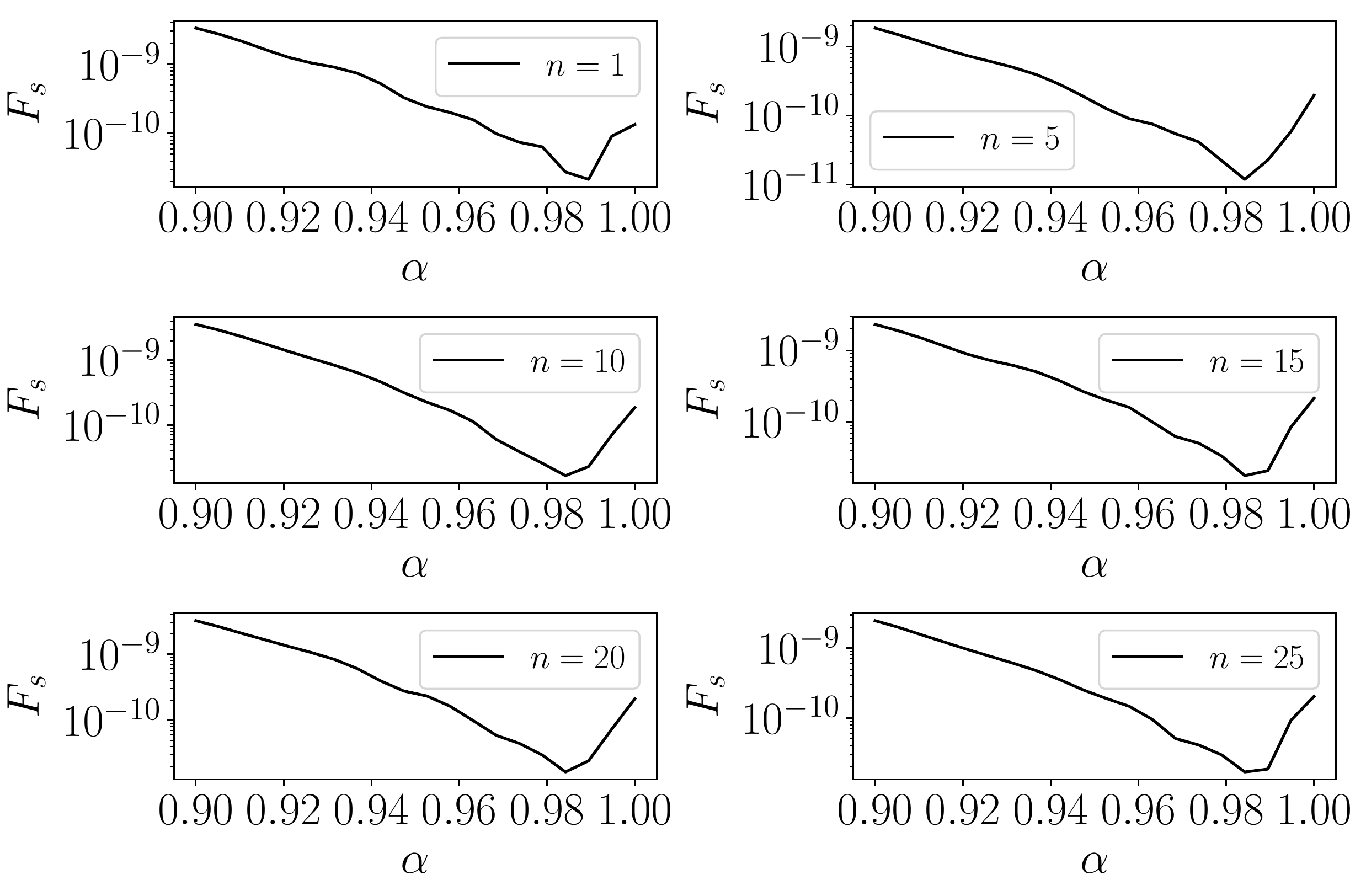}
\caption{Dependence of the stochastic convergence functional $F_s$ on the preconditioner parameter $\alpha$ for $K = 20$ and different values of initial guess vectors $n$. 
The test case is $m = 2500$, constant coefficients $D_{1,2}$.
The value $n=50$ used in the experiments results in a plot indistinguishable 
from the plot for $n=25$.}
\label{fig::n}
\end{figure}

\begin{figure}
\centering
\includegraphics[width=0.7\textwidth]{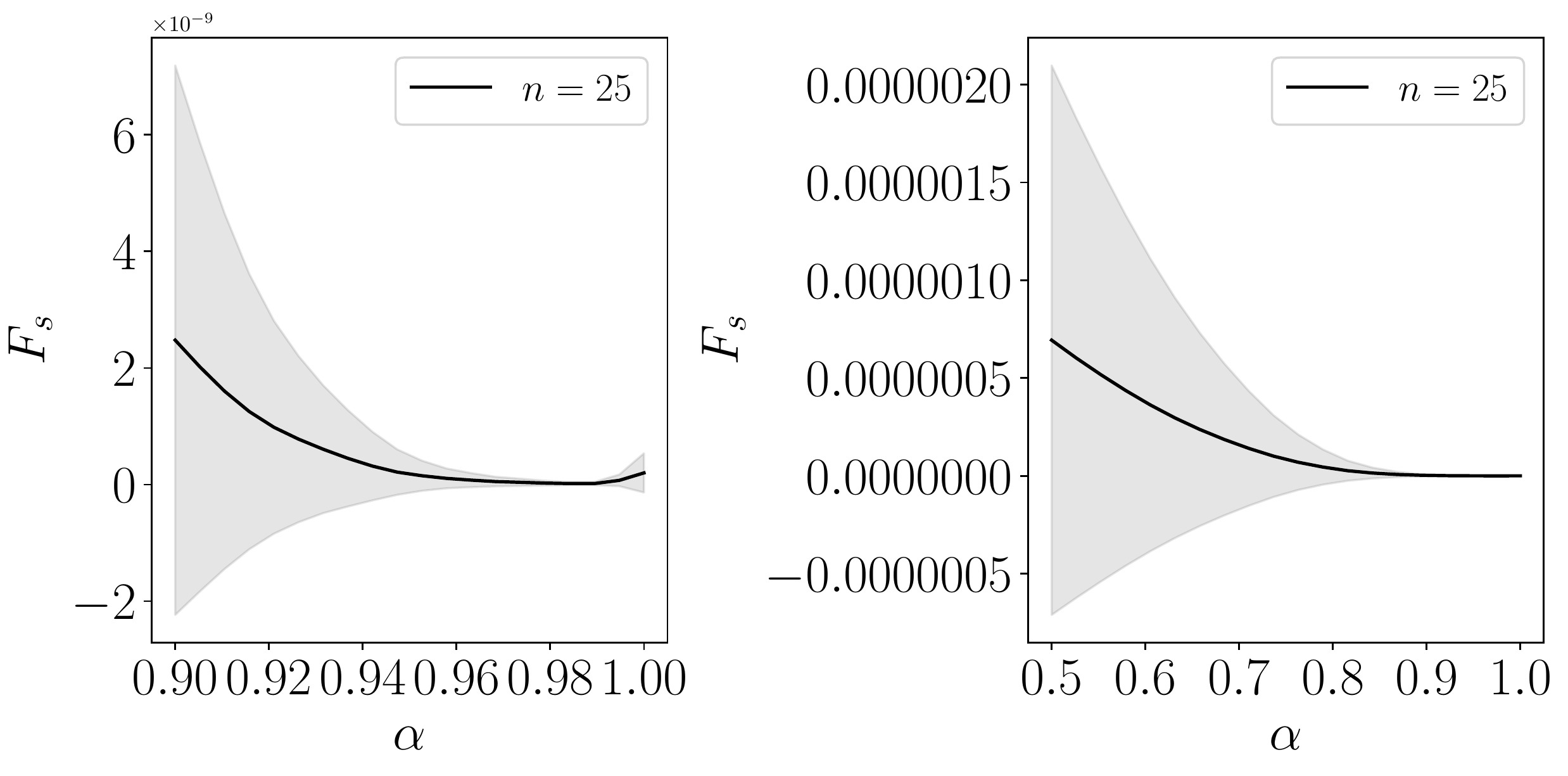}
\caption{Confidence interval (the gray area) of the stochastic convergence 
functional $F_s$ for $\alpha\in[0.9,1]$ (left) and $\alpha\in[0.5,1]$ (right).
The test case is $m = 2500$, constant coefficients $D_{1,2}$.}
\label{fig::conf_int}
\end{figure}

\subsection{Test problem 2}

As the second test problem a linear system with matrix \texttt{bcsstk16} from The SuiteSparse Matrix Collection~\cite{davis2011university} is taken of the size $m = 4884$.
We choose the right-hand side vector to have all its entries ones.

The main point of this test is to demonstrate that the proposed stochastic functional $F_s$ can be used to find parameter $\omega$ in the SSOR($\omega$) preconditioner, whenever the analytical formula~\eqref{om_opt} is not applicable.
The considered matrix is such that the Jacobi iterations diverge and therefore the analytical expression~\eqref{om_opt} can not be used.
Instead, we use the stochastic optimization procedure as discussed in Section~\ref{sec::test_problem1} with the number of preconditioned iterations $K = 15$ and number of random initial guess vectors $n=10$.
We search optimal parameter $\omega^*$ in interval $[0, 2]$, see~\cite{SaadBook2003}.
The results presented in Figure~\ref{fig::bcsstk16} show that both classical and stochastic functionals yield similar parameters values and undistinguished CG convergence plots.
However, using the classical functional $F_c$ is much more expensive than using the stochastic functional $F_s$, as eigenvalue computations are required for every evaluation of $F_c$.
Thus, with zero-order optimization method, e.g. Brent method, the proposed stochastic functional $F_s$ is more appropriate than $F_c$ to find unknown parameter in preconditioner SSOR($\omega$) for CG.

\begin{figure}[!h]
\centering
\begin{subfigure}[t]{0.45\textwidth}
\centering
\includegraphics[width=\textwidth]{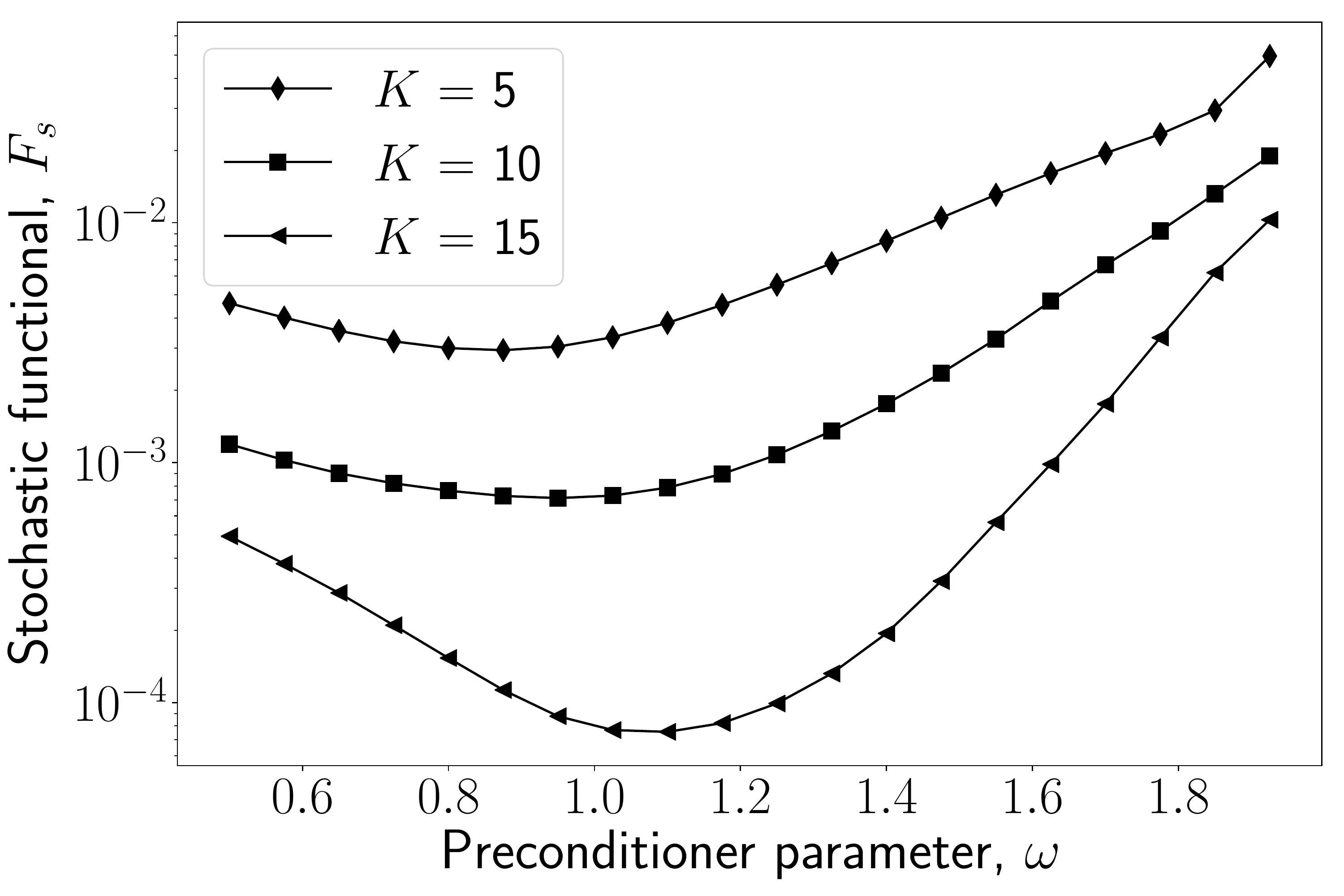}
\caption{}
\label{fig:stoch_bcsstk16_iter}
\end{subfigure}
~
\begin{subfigure}[t]{0.45\textwidth}
\centering
\includegraphics[width=\textwidth]{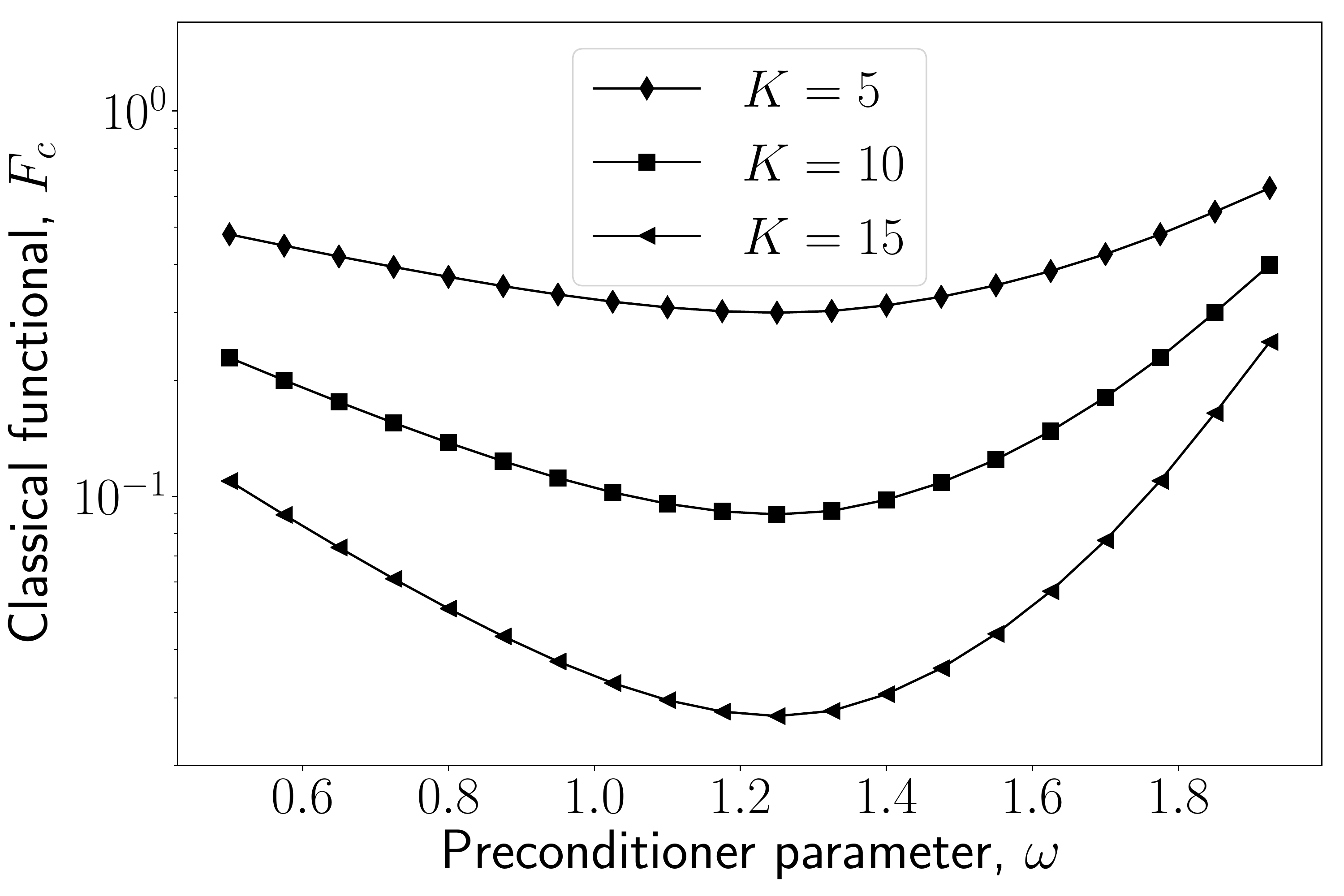}
\caption{}
\label{fig:class_bcsstk16_iter}
\end{subfigure}
\\
\begin{subfigure}[t]{0.45\textwidth}
\centering
\includegraphics[width=\textwidth,height=0.24\textheight]{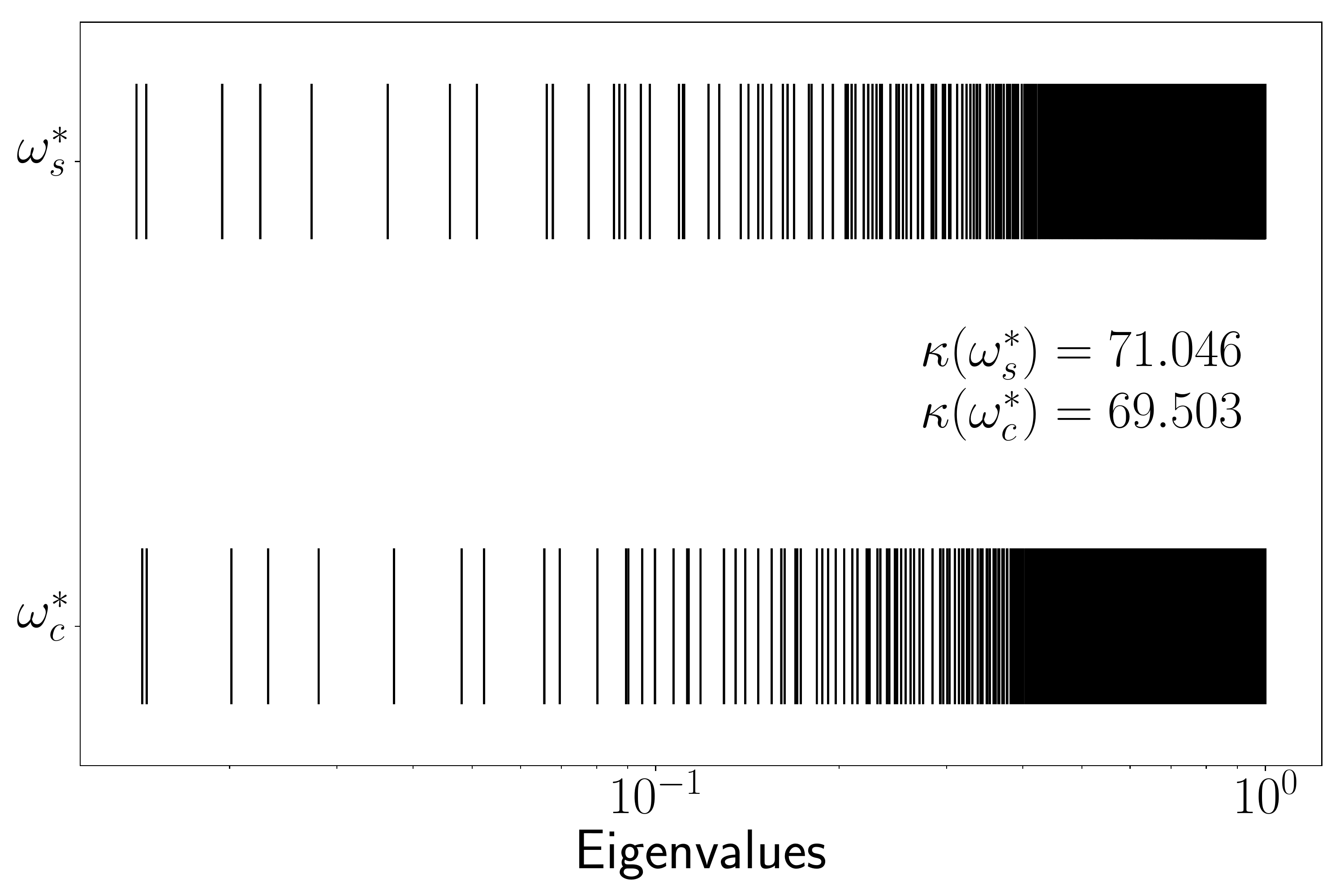}
\caption{}
\label{fig:stoch_class_bcsstk16_eigvals}
\end{subfigure}
~
\begin{subfigure}[t]{0.5\textwidth}
\centering
\includegraphics[width=\textwidth]{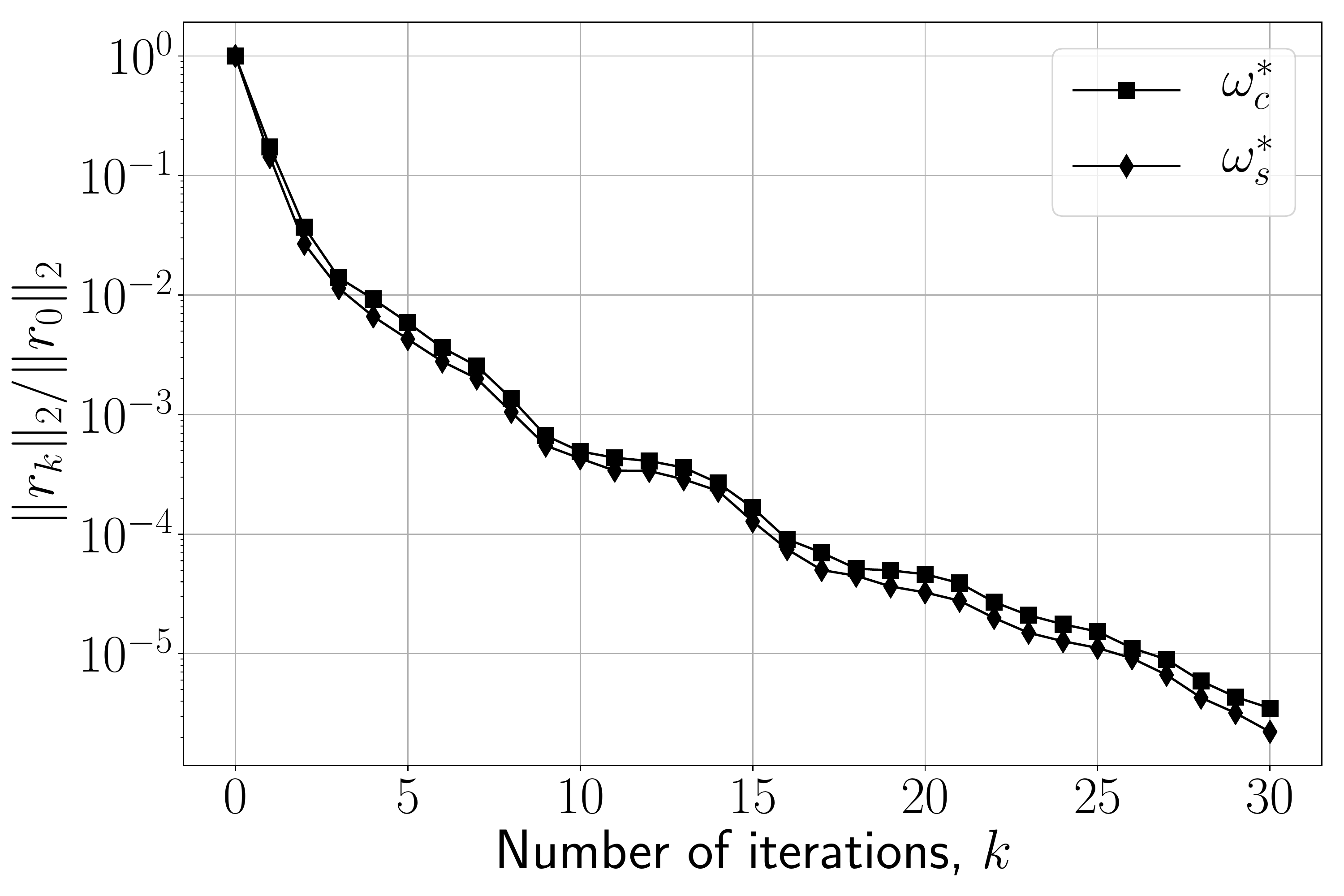}
\caption{}
\label{fig:conv_compare_bcsstk16}
\end{subfigure}
    
    \caption{Test case \texttt{bcsstk16}.
    Plots~(a),(b): dependence of the functionals $F_s$, $F_c$ on the preconditioner parameter $\omega$ for different iteration number $K$.
    Plot~(c): eigenvalues of the preconditioned matrix for the optimal values $\omega_s^*$ (based on~$F_s$) and $\omega_c^*$ (based on~$F_c$).
    Plot~(d): Residual norm convergence of CG preconditioned by SSOR($\omega$) for $\omega_s^*$ and $\omega_c^*$.}
    \label{fig::bcsstk16}
\end{figure}

\section{Conclusion}
\label{sect:concl}
In this paper, a stochastic approach to estimate convergence rate of iterative linear system solvers is presented.
Our estimate, which we call a stochastic 
convergence functional, is essentially based on monitoring the \emph{mean} convergence rate for a number of random initial guess vectors.  
For linear stationary iterative methods it is shown that the stochastic convergence functional coincides with the classical convergence estimate based on the spectral radius of the iteration matrix.  
For the CG method, which is a nonlinear nonstationary method, 
both analysis and experiments suggest that the stochastic convergence functional provides a sharper convergence measure than the classical estimate based on the spectral condition number of the system matrix.
We also show that the new stochastic convergence functional can be used for optimizing parameters in preconditioners for the CG method. 
Numerical tests for the CG method preconditioned by the $\mathrm{RIC}_{\alpha}(0)$ (relaxed incomplete Cholesky factorization with no fill in) and by the SSOR($\omega$) preconditioners are presented.  
The tests demonstrate that the new stochastic functional provides a better means for optimizing the preconditioner parameters than minimizing the spectral condition number.

Simple convergence analysis presented here shows that
the classical convergence estimate based on the 
spectral condition number can be improved for
some initial guess vectors.  An interesting open question remains whether other convergence estimates, in particular,
which demonstrate superlinear convergence, can be 
improved for some initial guess vectors.  We believe that
this is might be true and leave this for future work.

Another interesting extension of this work would be 
precondtioner optimization with respect to different parameters.
This is relevant, for instance, for circulant  preconditioners~\cite{chan1988optimal,oseledets2006unifying,tyrtyshnikov1992optimal}.
In this case some gradient optimization methods in combination with automatic differentiation tools (such as Autograd, Pytorch, etc.) can be successfully used, see our recent work~\cite{katrutsa2017deep}.

Finally, a relevant question is whether our stochastic
optimization procedure can be combined with solving
multiple linear systems by Krylov subspace 
recycling~\cite{Amritkar_ea2015,BennerFeng2011,OConnel2017}.
One could, for example, carry out optimization
based on the given (rather than on random) right hand side vectors,
starting off with an unoptimized preconditioner 
and carrying out optimization ``on the fly''.
We hope to explore this in a future work.


\bibliographystyle{siamplain}
\bibliography{lib}

\end{document}